\documentclass{IEEEtran}

\scrollmode

\usepackage{amsmath}
\usepackage{amsthm}
\usepackage{amssymb}
\usepackage{srcltx}
\usepackage{hyperref}
\usepackage{tikz}

\newtheorem{lemma}{Lemma}[section]
\newtheorem{theorem}[lemma]{Theorem}
\newtheorem{corollary}[lemma]{Corollary}
\newtheorem{proposition}[lemma]{Proposition}

\theoremstyle{definition}

\newtheorem{definition}[lemma]{Definition}

\numberwithin{equation}{section}

\newcommand{\leaveout}[1]{}

\makeatletter
\renewcommand{\p@enumii}{}
\makeatother

\newcommand{\DD}[1]{\mathbin{\frac{\rm d }{{\rm d }#1}}}
\newcommand{\ddt}{\DD t}

\newcommand{\ud}{\,{\mathrm d}}

\newcommand{\Div}{{\rm div\,}}

\newcommand{\Grad}{\nabla}
\newcommand{\supp}{{\rm supp\,}}

\newcommand\R{{\mathbb R}}
\newcommand\C{{\mathbb C}}

\newcommand\K{{\mathbb K}}
\newcommand\F{{\mathbb F}}

\newcommand\rplus{{\R_{+}}}
\newcommand\cplus{{\C_{+}}}

\newcommand\rminus{{\R_{-}}}
\newcommand\T{{\mathbb T}}
\newcommand\U{{\mathbb U}}
\newcommand\V{{\mathbb V}}

\newcommand\zero{\set{0}}

\newcommand{\Ascr}{\mathcal A}
\newcommand{\Bscr}{\mathcal B}
\newcommand{\Cscr}{\mathcal C}

\newcommand{\Gscr}{\mathcal G}
\newcommand{\Hscr}{\mathcal H}

\newcommand{\Lscr}{\mathcal L}

\newcommand{\Wscr}{\mathcal W}

\newcommand{\proj}{\mathbf P}
\newcommand{\shift}{\mathbf S}

\newcommand{\dom}[1]{\mathrm{dom}\left(#1\right)}
\newcommand{\res}[1]{\rho\left(#1\right)}
\newcommand{\re}[0]{\mathrm{Re}\,}

\newcommand{\Ipdp}[2]{\left\langle #1 , #2 \right\rangle}

\newcommand{\set}[1]{\left\lbrace #1 \right\rbrace}

\newcommand{\bigmid}{\bigm\vert}

\newcommand{\biggmid}{\biggm\vert}

\newcommand{\bi}{\begin{itemize}}
\newcommand{\ei}{\end{itemize}}
\newcommand{\be}{\begin{enumerate}}
\newcommand{\ee}{\end{enumerate}}

\newcommand{\Afrak}{\mathfrak A}

\newcommand{\Gfrak}{\mathfrak G}

\newcommand{\Pfrak}{\mathfrak P}

\newcommand{\Tfrak}{\mathfrak T}
\newcommand{\Ufrak}{\mathfrak U}

\newcommand{\sbm}[1]{\left[\begin{smallmatrix}#1\end{smallmatrix}\right]}

\newcommand{\bbm}[1]{\begin{bmatrix}#1\end{bmatrix}}

\makeatletter

\def\etv{& \hskip-.3em\vrule\hskip-.3em &} 
\def\smalletv{&\vrule&} 
\def\smallcrh{\vrule height0pt depth2\ex@ width0pt
\cr\noalign{\hrule}
\vrule height6.5\ex@ depth0pt width0pt}
\newbox\smallstrutbox
\setbox\smallstrutbox=\hbox{\vrule height6pt depth1.5pt width\z@}
\def\smallstrut{\relax\ifmmode\copy\smallstrutbox\else\unhcopy\smallstrutbox\fi}

\newenvironment{sysmatrix}{
\let\|=\etv
\hskip \arraycolsep
\begin{matrix}}
{\end{matrix}
\hskip \arraycolsep
}       

\newenvironment{smallsysmatrix}{\null\,\vcenter\bgroup
\let\|=\smalletv

\def\\{\smallstrut\math@cr}
\restore@math@cr\default@tag
\baselineskip\z@skip \lineskip\z@skip \lineskiplimit\lineskip
\ialign\bgroup\hfil$\m@th\scriptstyle##$\hfil&&\thickspace\hfil
$\m@th\scriptstyle##$\hfil\crcr
\crcr\noalign{\vskip -.3\ex@}%
}{\crcr\noalign{\vskip -.2\ex@}%
\crcr\egroup\egroup\,%
}

\makeatother

\begin{document}

\title{Well-posedness of time-varying linear systems}

\author{Mikael Kurula

\thanks{M. Kurula is with \AA bo Akademi University, Mathematics and Statistics, Domkyrkotorget 1, 20500 \AA bo, Finland (e-mail: mkurula@abo.fi).}}

\thispagestyle{empty}

\maketitle

\begin{abstract}
In this paper, we give easily verifiable sufficient conditions for two classes of perturbed linear, passive PDE systems to be well-posed, and we provide an energy inequality for the perturbed systems. Our conditions are in terms of smoothness of the operator functions that describe the multiplicative and additive perturbation, and here well-posedness essentially means that the time-varying systems have strongly continuous Lax-Phillips evolution families. A time-varying wave equation with a bounded multi-dimensional Lipschitz domain is used as illustration, and as a part of the example, we show that the time-invariant wave equation is a ``physically motivated'' scattering-passive system in the sense of Staffans and Weiss. The theory also applies to time-varying port-Hamiltonian systems.
\end{abstract}

\section{Introduction} 

Every linear, time-invariant, \emph{well-posed} system $\Sigma_i$ in continuous time, whose input space $U$, state space $X$ and output space $Y$ are Hilbert spaces, can after some technical setup be written in the following familiar-looking form:
\begin{equation}\label{eq:WPiCompat}
\Sigma_i:\quad\left\{
\begin{aligned}
	\dot x(t)&=A_{-1}x(t)+Bu(t),\\
	y(t)&=\overline Cx(t)+Du(t),\qquad t\geq\tau, \\
	x(\tau)&=x_\tau,
\end{aligned}\right.
\end{equation}
where $x(t)\in X$ is the \emph{state} at time $t$, $u(t)\in U$ is the \emph{input}, $y(t)\in Y$ is the \emph{output}, and $A_{-1}$ and $\overline C$ are certain extensions of the main operator $A$ and observation operator $C$. We will describe this class of systems in more detail below, in \S\ref{sec:prel}.

The system \eqref{eq:WPiCompat} is \emph{(scattering) passive} if all its trajectories satisfy the following energy inequality, for all $t\geq\tau$:
\begin{equation}\label{eq:LTIpasv}
	\|x(t)\|_X^2+\int_\tau^t\|y(s)\|_Y^2\ud s \leq 
	\|x(\tau)\|_X^2+\int_\tau^t\|u(s)\|_U^2\ud s.
\end{equation}

Let $\sbm{A_{-1}&B\\\overline C&D}$ be a time-invariant passive linear system on $(U,X,Y)$. In this paper, we prove that 
\begin{equation}\label{eq:Sigmal}
\Sigma_l:\quad\left\{
\begin{aligned}
	P(t)\,\dot x(t)&=\big(A_{-1}+P(t)G(t)\big)\,x(t)+Bu(t)\\
	y(t)&=\overline Cx(t)+Du(t),\qquad t\geq\tau,\\
	x(\tau)&=x_\tau,
\end{aligned}\right.
\end{equation}
and 
\begin{equation}\label{eq:Sigmar}
\Sigma_r:\quad\left\{
\begin{aligned}
	\dot x(t)&=\big(A_{-1}P(t)+G(t)\big)\,x(t)+Bu(t)\\
	y(t)&=\overline CP(t)x(t)+Du(t),\qquad t\geq\tau,\\
	x(\tau)&=x_\tau,
\end{aligned}\right.
\end{equation}
define time-varying well-posed systems (defined later) under certain smoothness conditions on $P(\cdot)$ and $G(\cdot)$.

The time-varying systems \eqref{eq:Sigmal} and \eqref{eq:Sigmar} can be written in many equivalent forms, by making different choices of $P(\cdot)$ and $G(\cdot)$, and some of these make the notation `l' for ``left'', `r' for ``right'' more evident; in $\Sigma_l$ we could for instance write
$$
	\dot x(t)=P(t)^{-1}\big(A_{-1}x(t)+Bu(t)\big)+G(t)\,x(t).
$$
The particular choice \eqref{eq:Sigmal}\,--\,\eqref{eq:Sigmar} has the advantage that the trajectories of both systems satisfy the same energy inequality: for all $t\geq\tau$ in the time interval $J$ of the system,
\begin{equation}\label{eq:SigmalEnergy}
\begin{aligned}
	&\Ipdp{P(t)x(t)}{x(t)}_X+\int_\tau^t \|y(s)\|_Y^2\ud s \leq \\ 
	&\qquad \Ipdp{P(\tau)x(\tau)}{x(\tau)}_X+\int_\tau^t \|u(s)\|_U^2\ud s\\
	&\qquad\quad +\int_\tau^t\Ipdp{\dot P(s)x(s)}{x(s)}_X\ud s\\
		&\qquad\quad +2\re\int_\tau^t\Ipdp{P(s)x(s)}{G(s)x(s)}_X\ud s.
\end{aligned}
\end{equation} 

Our investigation requires that we prove new generation results for evolution families, which is in itself a valuable contribution, since such results are currently rather scarce. The theory in the present paper generalizes the work of Schnaubelt and Weiss \cite{SchWe10}, and that of Chen and Weiss \cite{ChenWeiss15}, to a large extent by combining the techniques of these two papers. The exposition here is brief, avoiding duplication of detail through careful referencing. Reading \cite{SchWe10} and \cite{ChenWeiss15} first is recommended, to obtain needed background and many references.

Based on previous work \cite{KuZwWave}, we use the wave equation on a bounded Lipschitz domain $\Omega\subset\R^n$ to illustrate the applicability of the results of the present paper. This example is an extension of the wave equation in \cite[\S5]{SchWe10}, which can handle a moving object inside the domain, but the example cannot be treated with the tools developed in \cite{SchWe10}. As an intermediate step, we prove that the wave equation \eqref{eq:physPDEscatt} below can be written as a ``physically motivated'' scattering passive system in the sense of Staffans and Weiss \cite{StWe12b,StWe12a}. Due to the progress in \cite{KuZwWave} and \cite{StWe12b} after \cite{SchWe10}, the treatment of the example is somewhat easier in this paper than in \cite[\S5]{SchWe10}.

Linear port-Hamiltonian systems \cite{JaZwBook} are a large class of abstract PDEs with one-dimensional spatial domains, which includes the wave equation (on a string) and various beam equations. The theory in the present paper applies to time-varying port-Hamiltonian systems, in the same way as it applies to the wave equation in \S\ref{sec:Wave}; see in particular \cite[\S11.3]{JaZwBook} and \cite[(3.1) and Thm 4.6]{GZM05}. In fact, in closely related indpendent work \cite{JaLa19}, Jacob and Laasri consider well-posedness (defined sligthly differently) of time-varying boundary control systems, using port-Hamiltonian systems as motivating example. There is a considerable methodical overlap between the present paper and \cite{JaLa19}, and the theory in \cite{JaLa19} can likely also cover the example in \S V. Finally, we mention that Paunonen \cite{Pau17} and Pohjolainen \cite{PauPoh12} have studied robust output regulation of \emph{periodically} time-varying distributed parameter systems.

In \S\ref{sec:prel}, we collect the needed background on time-varying well-posed systems and their Lax-Phillips evolution families. Section \ref{sec:generation} contains our evolution-family generation results, in \S\ref{sec:WP} we prove well-posedness of \eqref{eq:Sigmal} and \eqref{eq:Sigmar}, and in \S\ref{sec:Wave} the paper is concluded with the wave equation example.

\section{Time-varying well-posed linear systems}\label{sec:prel}

In this section, we fix the notation and concepts needed later. We make the following assumptions throughout the paper: By $J\subset\R$ we denote a closed (time) interval of positive length, and we define $\Delta_J:=\set{(t,\tau)\in J^2\mid t\geq \tau}$ (a triangle if $J$ happens to be compact). We identify, e.g., $L^2(J;U)$ with the subspace of $L^2(\R;U)$ consisting of elements with support contained in $J$, and by $\proj_J$, we denote the orthogonal projection (by truncation) onto $L^2(J;U)$ in $L^2(\R;U)$. The bilateral shift of functions defined on $\R$ is $(\shift_tu)(\tau)=u(t+\tau)$, and we abbreviate $\shift_t^\pm:=\proj_{\R_\pm}\shift_t$, where $\R_\pm$ are in general closed or open, as fitting for the context. By writing, e.g., $H^1(J;U)$, we more precisely mean $H^1(J^0;U)$, where $J^0$ is the interior of $J$, and by derivatives evaluated at any end points of $J$, we mean the appropriate one-sided derivatives. 

\begin{definition}\label{def:evolfam}
A \emph{strongly continuous evolution family} on the Hilbert space $X$ with time interval $J$ is a two-parameter family $\T$ defined on $\Delta_J$, such that
\begin{enumerate}
\item $\T(t,\tau)\in\Lscr(X)$ for all $(t,\tau)\in\Delta_J$,
\item $\T(t,s)\,\T(s,\tau)=\T(t,\tau)$ for all $t,s,\tau\in J$: $t\geq s\geq \tau$,
\item $\T(t,t)=I$ for all $t\in J$, and
\item $(t,\tau)\mapsto\T(t,\tau)z$ is in $C(\Delta_J;X)$ for all $z\in X$.
\end{enumerate}

An evolution family $\T$ with time interval $J$ is \emph{locally (uniformly) exponentially bounded}  if for every compact $[a,b]\subset J$, there exist $M,\omega\in\R$, such that
\begin{equation}\label{eq:ExpBdd}
	\|\T(t,\tau)\|\leq Me^{\omega (t-\tau)},\qquad
		(t,\tau)\in\Delta_{[a,b]}.
\end{equation}
If there exist $M,\omega\in \R$ with the above property, which are independent of $[a,b]$, then we call $\T$ \emph{exponentially bounded}.

A family $\set{A(t): X\supset\dom{A(t)}\to X\mid t\in J}$ of $C_0$-semigroup generators are said to \emph{generate} $\T$ if
\begin{enumerate}
\item[a)] $\T(t,\tau)\,\dom{A(\tau)}\subset\dom{A(t)}$ for all $(t,\tau)\in \Delta_J$,
\item[b)] for every $\tau\in J$ with $\tau<\sup\, J$ and every $x_\tau\in\dom{A(\tau)}$, the function 
$$
	x(t):=\T(t,\tau)x_\tau,\quad t\in J_\tau:=\set{t\in J\mid t\geq\tau},
$$
is a solution in $C^1(J_\tau;X)$ of the Cauchy problem
\begin{equation}\label{eq:Cauchy}
	\dot x(t)=A(t)\,x(t), \quad t\in J_\tau,\quad x(\tau)=x_\tau.
\end{equation}
\end{enumerate}
\end{definition}

Not all evolution families have generator families in this sense, but a generator family can generate at most one evolution family. If $\T(t+s,s)$ is independent of $s\in J$ for all $t\geq0$ such that $t+s\in J$, then $\T^i_t:=\T(t+s,s)$, $t\geq0$ and $t+s\in J$, can be extended to a unique $C_0$ semigroup on $X$, which is always exponentially bounded, $\|\T^i_t\|\leq Me^{\omega t}$ for some $M,\omega\in\R$.

Some of our proofs use duality arguments, and we then need \emph{backward} evolution families \cite[Def.\ 2.5]{SchWe10}. Such are two-parameter families $\T$ defined on $\Delta_J$ with properties 1), 3), 4) in Def.\ \ref{def:evolfam}, but 2) is replaced by
$$
	\T(s,\tau)\,\T(t,s)=\T(t,\tau),\qquad t,s,\tau\in J,~t\geq s\geq \tau.
$$
A $C_0$-semigroup generator family $A(t)$, $t\in J$, \emph{generates} a backward evolution family $\T$ if 
$$
	\T(t,\tau)\,\dom{A(t)}\subset\dom{A(\tau)}
$$ 
for all $(t,\tau)\in \Delta_J$ and for every $t\in J$ with $t>\inf\, J$ and every $x_t\in\dom{A(t)}$, the function 
$$
	x(\tau):=\T(t,\tau)x_t,\quad \tau\in J\cap(-\infty,t]
$$
is a strongly continuous (in $X$) solution of the backward-time Cauchy problem 
$$
	\dot x(\tau)=-A(\tau)\,x(\tau), \quad \tau\in J\cap(-\infty,t],\quad x(t)=x_t.
$$

A \emph{contraction semigroup} is exponentially bounded by $M=1$ and $\omega=0$, and the generator of a contraction semigroup is maximal dissipative on $X$, meaning that $\re\Ipdp{Ax}x\leq 0$ for all $x\in\dom A$ and the resolvent set $\res A$ contains the open complex right-half plane $\cplus$.

We will prove that \eqref{eq:Sigmal} and \eqref{eq:Sigmar} describe time-varying  well-posed linear systems in the sense of \cite[Def.\ 3.2]{SchWe10}:

\begin{definition}\label{def:WP}
A (time-varying) \emph{well-posed system} $\Sigma$ on a closed time interval $J\subset\R$, with Hilbert input, state and output spaces $(U,X,Y)$, is a quadruple of linear operator families defined for $(t,\tau)\in \Delta_J$, mapping
\begin{equation}\label{eq:WPmaps}
\begin{aligned}
	\T(t,\tau)&:X\to X, ~ \F(t,\tau):L^2(J;U)\to L^2(J;Y),\\
	\quad\Phi(t,\tau)&:L^2(J;U)\to X,~  \Psi(t,\tau):X\to L^2(J;Y),
\end{aligned}
\end{equation}
boundedly, which have the following additional properties:
\begin{enumerate}
\item $\T$ is an evolution family on $X$ with time interval $J$,
\item the other families are \emph{causal} in the sense that
\begin{equation}\label{eq:WPcausal}
\begin{aligned}
	\Phi(t,\tau)&=\Phi(t,\tau)\proj_{[t,\tau]},\\ 
	\Psi(t,\tau)&=\proj_{[t,\tau]}\Psi(t,\tau)\qquad \text{and} \\
	\F(t,\tau)&=\proj_{[t,\tau]}\F(t,\tau) =\F(t,\tau)\proj_{[t,\tau]},
\end{aligned}
\end{equation}
\item all four families are locally uniformly bounded,
\item and they encode the linearity of the system, so that for all $t,s,\tau\in J$ with $t\geq s\geq \tau$:
\begin{equation}\label{eq:WPlin}
\begin{aligned}
	\Phi(t,\tau) &= \Phi(t,s)+\T(t,s)\Phi(s,\tau), \\
	\Psi(t,\tau) &= \Psi(t,s)\T(s,\tau)+\Psi(s,\tau), \\
	\F(t,\tau) &= \F(t,s)+\F(s,\tau) + \Psi(t,s)\Phi(s,\tau).
\end{aligned}
\end{equation}
\end{enumerate}

A well posed system is called \emph{time invariant} if $J=\rplus$ and the following are all independent of $s\geq0$, for $t\geq0$:
$$
\begin{aligned}
	\T_t^i&:=\T(t+s,s),\qquad\quad \,\F_t^i:=\shift_s\,\F(t+s,s)\,\shift_{-s},\\
	\Phi_t^i&:=\Phi(t+s,s)\,\shift_{-s},\quad \Psi_t^i:=\shift_s\Psi(t+s,s).
\end{aligned}
$$
\end{definition}

The four operator families of a well-posed system $\sbm{\T&\Phi\\\Psi&\F}$ are strongly continuous in $X$; see \cite[Prop.\ 3.5]{SchWe10}. By a \emph{trajectory} of a well-posed system on $J$ with initial state $x_\tau$ at time $\tau\in J$, $\tau<\sup\,(J)$, and input $u\in L_{loc}^2(J_\tau;U)$, we mean the triple $(u,x,y)\in L_{loc}^2(J_\tau;U)\times C(J_\tau;X)\times L_{loc}^2(J_\tau;Y)$:
$$
\begin{aligned}
	x(t)&=\T(t,\tau)x_\tau+\Phi(t,\tau)u\qquad \text{and} \\
	\proj_{[\tau,t]}y&=\Psi(t,\tau) x_\tau+\F(t,\tau)u,\qquad t\in J_\tau;
\end{aligned}
$$
see \cite[p.\ 282]{SchWe10}. In particular, a trajectory of a well-posed system is uniquely determined by its initial state $x_\tau$ and input $u$. \emph{Classical trajectories} of $\sbm{\T&\Phi\\\Psi&\F}$ are such that $(u,x,y)\in C(J_\tau;U)\times C^1(J_\tau;X)\times C(J_\tau;Y)$. 

The class of time-invariant well-posed systems in Def.\ \ref{def:WP} coincides with the standard class; take $s=0$ in the definitions of the time-invariant operators and see \cite[\S2.8]{StafBook}.

In this paper, we will be interested in evolution family generator families of the form $A_l(t)=P(t)^{-1}A+G(t)$ and $A_r(t)=A\,P(t)+G(t)$, where $A$ generates a contraction semigroup on $X$. By the Lumer-Phillips theorem \cite[II.3.5]{EnNa00}, $A$ is maximal dissipative.

We will apply the generation results in \S\ref{sec:generation} to the generator of the Lax-Phillips \emph{semigroup}
\begin{equation}\label{eq:LPsg}
	\Tfrak_t^i:=\bbm{\shift_t^-&0&0 \\ 0&I&0 \\ 0&0&\shift_t^+}
	\bbm{I&\Psi_t^i&\F_t^i \\ 0&\T_t^i&\Phi_t^i \\ 0&0&I },\quad t\geq0,
\end{equation}
of a \emph{passive} time-invariant system $\sbm{\T_t^i&\Phi_t^i\\\Psi_t^i&\F_t^i }$. Due to passivity, this is a contraction semigroup on the space $\Hscr:=L^2(\rminus;Y)\times X\times L^2(\rplus;U)$ with generator \cite[Prop.\ 3.1]{SchWe10}
\begin{align}
	\Afrak(y,x_0,u) &=(y',A_{-1}x_0+Bu(0),u'), \label{eq:LPsgGen} \\
	\dom{\Afrak}&=\big\{(y,x_0,u) \in H^1(\rminus;Y)\times X\times H^1(\rplus;U) \bigmid \nonumber\\
		&\quad A_{-1}x_0+Bu(0)\in X, \ y(0)=\overline Cx_0+Du(0)\big\}.\nonumber
\end{align}
The resulting operator family will be associated to a time-varying Lax-Phillips \emph{evolution family}, which uniquely determines a time-varying well-posed system. 

Proposition 3.7 in \cite{SchWe10} is the key to the approach:

\begin{theorem}\label{thm:LPsys}
Let $\T$, $\Phi$, $\Psi$ and $\F$ be two-parameter families of linear operators defined on some $\Delta_J$, mapping as in \eqref{eq:WPmaps} and having the causality properties \eqref{eq:WPcausal}. These families form a well-posed system if and only if the (Lax-Phillips) family
\begin{equation}\label{eq:LPevol}
	\Tfrak(t,\tau):=\bbm{
	\shift_{t-\tau}^-&\shift_{t}^-\,\Psi(t,\tau)&\shift_t^-\,\F(t,\tau)\,\shift_{-\tau} \\
	0&\T(t,\tau)&\Phi(t,\tau)\,\shift_{-\tau} \\ 0&0&\shift_{t-\tau}^+},
\end{equation}
defined for $(t,\tau)\in \Delta_J$, is an evolution family on $\Hscr$.
\end{theorem}

\noindent
See \cite[\S3]{SchWe10} for more details on the connection between Lax-Phillips evolution families and well-posed systems.

The resolvent set $\res A$ of the generator $A$ of a contraction semigroup contains 1, by the Lumer-Phillips theorem, so we may equip $\dom A$ with the norm $\|x\|_{1}:=\|(I-A)x\|_X$ to make it a Hilbert space which is densely and continuously embedded in $X$; this space is commonly denoted by $X_1$. Moreover, we define the \emph{extrapolation space} $X_{-1}$ to be the completion of $X$ in the norm $\|x\|_{-1}:=\|(I-A)^{-1}x\|_X$. With this setup, $X_{-1}$ can be identified with the dual of $X_1^d:=\dom {A^*}$ with pivot space $X$, so that
\begin{equation}\label{eq:pivot}
	\Ipdp xz_{X_{-1},X_1^d}=\Ipdp xz_X, \quad x\in X,\,z\in X_1^d.
\end{equation}
Furthermore, the unitary operator $I-A:\dom A\to X$ can be uniquely extended into a unitary operator $I-A_{-1}:X\to X_{-1}$, where $A_{-1}$ is the unique extension of $A$ to an operator in $\Lscr(X;X_{-1})$. Then $A$ has the maximality property that $\dom A=\set{x\in X\mid A_{-1}x\in X}$.

With this setup, well-posed systems on Hilbert spaces can always be written in the form \eqref{eq:WPiCompat}, since they are \emph{compatible} by \cite[Thm 5.1.12]{StafBook}. The operator $\overline C\in\Lscr(Z;Y)$ is called a \emph{compatible extension} of the observation operator $C$ to the \emph{solution space} 
\begin{equation}\label{eq:solspace}
	Z:=\dom A+(\alpha-A_{-1})^{-1}BU, 
\end{equation}
for some $\alpha$ in the resolvent set $\res A$, where the particular choice of $\alpha$ does not matter. By \cite[Lemma 4.3.12]{StafBook}, $Z$ is a Hilbert space with
\begin{equation}\label{eq:Znorm}
	\|x\|_Z^2=\inf_{Ax+Bu\in X} \|Ax+Bu\|_X^2+\|x\|_X^2+\|u\|_U^2 ,
\end{equation}
and $(\alpha-A_{-1})^{-1}B\in\Lscr(U;Z)$. We have $\dom A\subset Z\subset X$ with continuous embeddings, and $Z\subset X$ is dense, but in general $\dom A\subset Z$ is not dense, so that $\overline C$ is in general not uniquely determined by $C$. However, $\overline C$ is uniquely determined by the system and $D$. See \cite[\S5.1]{StafBook} for more details.

By \cite[Prop.\ 5.2]{MaStWe06} or \cite[Thm 11.1.5]{StafBook}, a time-invariant, well-posed system $\Sigma_i$ is passive if and only if
\begin{equation}\label{eq:passpow}
	2\re\Ipdp{A_{-1}x+Bu}{x}\leq \|u\|^2-\|\overline Cx+Du\|^2
\end{equation}
for all $x\in X$ and $u\in U$ such that $A_{-1}x+Bu\in X$. We say that $\Sigma_i$ is \emph{energy preserving} if all continuous trajectories satisfy \eqref{eq:LTIpasv} with equality, and this is equivalent to \eqref{eq:passpow} holding with equality.

\section{Time-varying perturbation of maximal dissipative operators}\label{sec:generation}

The presentation in this section follows \cite{SchWe10}, with some ingredients added from \cite{ChenWeiss15}. We introduce two functions $P,G:J\to \mathcal L(X)$ which we use to perturb a maximal dissipative operator on $X$. Throughout the paper, these functions have the following properties, for all $t\in J$ and $z\in X$:
\begin{equation}\label{eq:standing}
\begin{aligned}
&\bullet P(t)=P(t)^*\geq0, \\
&\bullet \text{$P(t)$ has an inverse in $\Lscr(X)$},\\
& \bullet P(\cdot)z,\, P(\cdot)^{-1}z\in C^1(J;X)\text{ for all $z\in X$, and}\\
&\bullet G(\cdot)z\in C(J;X)\text{ for all $z\in X$}.
\end{aligned}
\end{equation}
From these assumptions and the uniform boundedness principle, see \cite[Thm 2.6]{Rudin73} or \cite[Thm 1.1.11]{TanabeBook}, it follows that $P(\cdot)$, $P(\cdot)^{-1}$, $\dot P(\cdot)$, and $G(\cdot)$ are all uniformly bounded on compact subintervals of $J$. Moreover,
\begin{equation}\label{eq:PinvDiff}
	\ddt P(t)^{-1}z=-P(t)^{-1}\dot P(t)P(t)^{-1}z,\qquad t\in J.
\end{equation}

\begin{theorem}\label{thm:generationL}
In addition to \eqref{eq:standing}, let $G(\cdot)z\in C^1(J;X)$ for all $z\in X$. Then
\begin{equation}\label{eq:AlDef}
\begin{aligned}
	A_l(t)&:=P(t)^{-1}A+G(t),\\
	\dom{A_l(t)}&:=\dom A,
\end{aligned} \qquad t\in J,
\end{equation}
generates an evolution family $\T_l$ on $X$ with time interval $J$, so that for all $(t,\tau)\in\Delta_J$ and $x_\tau\in\dom A$, $\T_l(t,\tau)x_\tau\in\dom A$ and $x(t):=\T_l(t,\tau)x_\tau$ solves the Cauchy problem 
\begin{equation}\label{eq:CauchyL}
	\dot x(t)=A_l(t)\,x(t), \quad t\in J_\tau,\quad x(\tau)=x_\tau.
\end{equation}

The restriction of $\T_l$ to a compact $\Delta_{[a,b]}$ has an exponential bound which depends only on the maximal values of $\|P(t)^{-1}\|$, $\|P(t)\|$ and $\|G(t)\|$ on $[a,b]$.

For all $x_0\in\dom{A}$, the function
\begin{equation}\label{eq:AlTl}
	(t,\tau)\mapsto A_l(t)\,\T_l(t,\tau)x_0,
\end{equation}
is in $C(\Delta_J;X)$. 

For all $t\in J$ with $t>\inf\,J$ and $x_0\in\dom{A_l}$, the function $\tau\mapsto \T_l(t,\tau)x_0$ is continuously differentiable in $X$, on $\set{\tau\in J\mid \tau\leq t}$, and 
\begin{equation}\label{eq:TlDtau}
	\frac{\partial}{\partial\tau} \T_l(t,\tau)x_0=-\T_l(t,\tau)A_l(\tau)x_0.
\end{equation}
\end{theorem}

In general the exponential bounds of $\T_l$ depend on the compact subinterval $[a,b]\subset J$, but if $P(\cdot)^{-1}$, $P(\cdot)$ and $G(\cdot)$ are uniformly bounded on all of $J$, then $\T_l$ is \emph{globally} exponentially bounded, rather than only locally.

\begin{proof}
We temporarily restrict $t$ and $\tau$ to a compact subinterval $[a,b]\subset J$ and throughout we assume that $t>a$. By \cite[p.\ 271]{SchWe10}, $P(t)^{-1}A$ generates a $C_0$-semigroup on $X$ for every fixed $t$, and the family $t\mapsto P(t)^{-1}A$ is a stable family of semigroup generators in the sense of \cite[Def.\ 5.2.1]{PazyBook}, with stability constants that depend only on the maximal values of $\|\dot P(t)\|$ and $\|P(t)^{-1}\|$ on $t\in[a,b]$. Since $\|G(t)\|\leq K$ for some $K$ independent of $t\in[a,b]$, it follows from \cite[Thm 5.2.3]{PazyBook} that $A_l(t)$, $t\in[a,b]$, is also a stable family of semigroup generators, whose stability bound $(M,\omega)$ further depends on the bound $K$ of $\|G(\cdot)\|$ on $[a,b]$. By \cite[Thm 5.4.8]{PazyBook}, the evolution family generated by $A_l(\cdot)$ on $[a,b]$ has the same exponential bound $(M,\omega)$, and it satisfies all the assertions with $J$ replaced by $[a,b]$ and \eqref{eq:AlTl} replaced by
$$
	(t,\tau)\mapsto \T_l(t,\tau)x_\tau \in C(\Delta_{[a,b]};X_1).
$$
From this follows, however, that $(t,\tau)\mapsto A\T_l(t,\tau)x_\tau \in C(\Delta_{[a,b]};X)$, and by the uniform boundedness of $P(\cdot)^{-1}$ on $[a,b]$, \eqref{eq:AlTl} holds with $J$ replaced by $[a,b]$.

Every fixed pair $(t,\tau)\in\Delta_J$ is contained in some $\Delta_{[a,b]}$ with $[a,b]$ compact. Therefore, the family $\T_l(t,\tau)$ defined for every $(t,\tau)\in\Delta_J$ as the evolution family generated by $A_l(\cdot)$ with the time interval $[\tau,t]$, has properties 1)--3) and a) in Def.\ \ref{def:evolfam}. Every point $(t_0,\tau_0)\in\Delta_J$, where we may want to verify continuity in condition 4) or \eqref{eq:AlTl}, is also contained in some compact $\Delta_{[a,b]}$. This proves in particular that $\T_l$ is an evolution family on all of $J$. Similarly, one verifies that condition b) in Def.\ \ref{def:evolfam} holds on all of $J$, that $\T_l(t,\cdot)x_0\in C^1(J;X)$, and that the latter satisfies \eqref{eq:TlDtau} for all $\tau\in J$ with $\tau\leq t$. 
\end{proof}

In order to drop the extra assumption that $G$ is strongly in $C^1$, we define the $\Lscr(X)$-valued averaged function $G_n$ by
\begin{equation}\label{eq:PnDef}
	G_n(t)z:=n\int_t^{t+1/n} G(s)z\ud s,\qquad t\in J, ~z\in X,
\end{equation}
where we extend $G(t):=G(b)$, $t>b$, if $J$ has a finite right endpoint $b$; then $G_n(\cdot)z\in C^1(J;X)$ for all $z\in X$ and $\|G(t)\|\leq K$ for $t$ in a compact interval $[a,b+1/n]$ implies that $\|G_n(t)\|\leq K$ for $t\in[a,b]$.

The dual $X_{-1,l}^{t}$ of $\dom{A^*P(t)^{-1}}$ with pivot space $X$ can be identified with the extrapolation space of $P(t)^{-1}A$ in Thm \ref{thm:generationL}. Moreover, the operators $P(t)^{-1}$ extend to locally uniformly bounded isomorphisms $P(t)^{-1}_{-1}$ from $X_{-1}$ to $X_{-1,l}^t$ with locally uniformly bounded inverses by \cite[Prop. 4.2(a)]{SchWe10}, so that we may identify $X_{-1,l}^t$ with $X_{-1}$. In practice, however, it is often easiest to use the time-varying norm 
$$
	\|x\|_{-1}=\|(I-P(t)_{-1}^{-1}A_{-1})^{-1}x\|_X
$$ 
on $X_{-1}$, where $A_{-1}$ is the extension of $A$ to an operator in $\Lscr(X;X_{-1})$. Then \eqref{eq:pivot} holds with $X_1^d$ replaced by $X_{l,1}^{d,t}:=\dom{A^*P(t)^{-1}}$. 

By Thm \ref{thm:generationL}, the family
\begin{equation}\label{eq:AnDef}
\begin{aligned}
	A_n(t)&:=P(t)^{-1}A+G_n(t), \\
	\dom{A_n(t)}&:=\dom A,\qquad t\in J,
\end{aligned}
\end{equation}
generates an evolution family $\T_n$ on $X$ with time interval $J$, whose restrictions to compact intervals $[a,b]\subset J$ have exponential bounds which are independent of $n$. With the setup in the preceding paragraph, $A_l(t)$ and $A_n(t)$ have unique extensions to operators in $\Lscr(X;X_{-1})$, and we denote these extensions by $A_{-1,l}(t)$ and $A_{-1,n}(t)$, observing that 
\begin{equation}\label{eq:Amin1lexp}
	A_{-1,l}(t)=P(t)^{-1}_{-1}A_{-1} +G(t),\quad t\in J,
\end{equation}
and analogous for $A_{-1,n}(t)$.

\begin{theorem}\label{thm:extensionL}
For all $(t,\tau)\in\Delta_J$ and $z\in X$, the limit
\begin{equation}\label{eq:TlDefTK}
	\T_l(t,\tau)z:=\lim_{n\to\infty} \T_n(t,\tau)z
\end{equation}
exists, with uniform convergence on $\Delta_{[a,b]}$ for compact $[a,b]\subset J$, and $\T_l$ is the unique locally exponentially bounded evolution family on $X$ with time interval $J$, which satisfies
\begin{equation}\label{eq:TlIntegEq}
	\T_l(t,\tau)z=\U(t,\tau)z+
	\int_\tau^t\U(t,s)\,G(s)\,\T_l(s,\tau)z\ud s
\end{equation}
for all $z\in X$ and $(t,\tau)\in\Delta_J$, where $\U$ is the locally exponentially bounded evolution family generated by $P(\cdot)^{-1}A$. 

If $G(\cdot)z\in C^1([a,b];X)$ for all $z\in X$, then $\T_l$ equals the evolution family in Thm \ref{thm:generationL}.

For every $\tau\in J$ with $\tau<\sup\,J$ and $x_\tau\in X$,
$$
	x(t):=\T_l(t,\tau)x_\tau,\qquad t\in J_\tau,
$$
is a solution in $C(J_\tau;X)$ of the equation
\begin{equation}\label{eq:TlMild}
	x(t)-x_\tau=\int_\tau^t A_{-1,l}(s)\,x(s)\,ds, \quad (t,\tau)\in\Delta_J,
\end{equation}
where the integral is computed in $X_{-1}$.

A function $z\in C(J_\tau;X)$ solves \eqref{eq:TlMild} if and only if it is in $C(J_\tau;X)\cap C^1(J_\tau;X_{-1})$ and solves the following Cauchy problem in $X_{-1}$:
\begin{equation}\label{eq:CauchyLextr}
		\dot x(t)=A_{-1,l}(t)\,x(t), \quad t\in J_\tau,\quad x(\tau)=x_\tau.
\end{equation}
\end{theorem}

Uniqueness in \eqref{eq:TlIntegEq} holds in the following stronger sense: Assume that $\mathbb S(t,\tau)$ is a two-parameter family of operators in $\Lscr(X)$, such that $\mathbb S(\cdot,\cdot)z$ is weakly continuous for all $z\in X$, and \eqref{eq:TlIntegEq} holds with $\mathbb S$ in place of $\T_l$ with weak integrals, for all $x_\tau\in X$ and $(t,\tau)\in\Delta_J$. Then $\mathbb S=\T_l$.

Surprisingly, proving uniqueness of the solutions of \eqref{eq:TlMild} and \eqref{eq:CauchyLextr} turned out tricky. The proof of \cite[Thm 3.9]{ChenWeiss15} uses properties of $C_0$-semigroups which may have no counterpart for evolution families in general. We will return to this later, in a remark immediately after Thm \ref{thm:Sigmal}.

\begin{proof}
We again fix a compact subinterval $[a,b]\subset J$. 

\emph{Step 1: Constructing $\T_l$ by applying a Trotter-Kato theorem to the Howland evolution semigroup of $\T_n$.} With the supremum norm, $C([a,b];X)$ is a Banach space. For $\sigma\geq0$ and $t\in [a,b]$, define
$$
	\sigma\mapsto (E_n^\sigma f)(t):=\left\{\begin{aligned}
		 \T_n(t,t-\sigma)\,f(t-\sigma),&\quad t-\sigma\in [a,b], \\
		 \T_n(t,a)\,f(a),&\quad \text{otherwise}.
	\end{aligned}\right.
$$
By showing that the function $(t,\sigma)\mapsto (E_n^\sigma f)(t)$ is in $C( [a,b]\times\rplus;X)$, hence uniformly continuous on $[a,b]\times[0,1]$, one can show that $\sigma\mapsto E_n^\sigma$, $\sigma\in\rplus$, is a $C_0$-semigroup on $C([a,b];X)$, which is called the (Howland) evolution semigroup \cite{How74} associated to $ \T_n$. An evolution family is uniquely determined by its Howland semigroup, because
\begin{equation}\label{eq:HowlInv}
	\T_n(t,\tau)f(\tau)=(E_n^{t-\tau}f)(t),\quad (t,\tau)\in\Delta_{[a,b]},
\end{equation}
for all $f\in C([a,b];X)$, and for all $z\in X$ there exists some continuous $f$ with $f(\tau)=z$. Moreover, $\|E_n^\sigma\|\leq Me^{\omega\sigma}$ if $(M,\omega)$ is an exponential bound for $\T_n$ on $[a,b]$.

The argument in the proof of \cite[Thm 3.12]{ChLa99}, see also \cite[Prop.\ 3.4]{ChenWeiss15}, can be slightly adapted to establish that the generator of $E_n$ is the unique closure $\Gscr_n$ (as an unbounded operator on $C([a,b];X)$) of 
\begin{equation}\label{eq:Gscrn}
\begin{aligned}
	& (\Gscr_{0,n}f)(t) :=A_n(t)f(t)-f'(t),~ t\in [a,b],\\
	 &\dom{\Gscr_{0,n}}:=\{f\in C^1([a,b];X)\mid \forall t\in [a,b]:\\
	 	& \qquad f(t)\in\dom A,~ A_n(\cdot)f(\cdot)\in C([a,b];X)\}.
\end{aligned}
\end{equation}
In particular, the denseness in $C([a,b];X)$ of the linear span of functions $t\mapsto\alpha(t)\T_n(t,\tau)x_\tau$ for $t>\tau$ and $t\mapsto0$ for $t\leq\tau$, where $\alpha\in C^1(\R)$ with compact support contained in $(\tau,\infty)$, follows easily from the denseness of this span in $C_0(\R;X)$, the space of continuous $X$-valued functions tending to zero at $\pm\infty$. Due to the continuity and the uniform boundedness of $P(\cdot)$, $P(\cdot)^{-1}$ and $G_n(\cdot)$ on $[a,b]$, the condition $A_n(\cdot)f(\cdot)\in C([a,b];X)$ is equivalent to $Af(\cdot)\in C([a,b];X)$, and hence the core $\dom{\mathcal G_{0,n}}$ is independent of $n$.

As in the proof of \cite[Prop.\ 3.7]{ChenWeiss15}, the uniform boundedness of $G$ on $[a,b]$ implies that the pointwise multiplication operators $M_G$ and $M_{G_n}$ are bounded on $C([a,b];X)$, and that $M_{G_n}f\to M_Gf$ for all $f\in C([a,b];X)$, so that $\Gscr_n f\to \Gscr f$ in $C([a,b];X)$ for all $f$ in the core $\dom{\Gscr_{0,n}}$, where
$$
	\Gscr:=\Gscr_n+M_G-M_{G_{n}}. 
$$
Proceeding as in step 2 of the proof of \cite[Thm 3.8]{ChenWeiss15}, we get that $\Gscr$ generates a strongly continuous semigroup $E$ on $C(X;[a,b])$ and that $\|E_n^\sigma f- E^\sigma f\|_\infty\to0$ as $n\to\infty$ for all $f\in C([a,b];X)$ and $\sigma\geq0$, where $\|\cdot\|_\infty$ is the maximum norm on $[a,b]$, uniformly for $\sigma$ restricted to compact intervals. 

If $G(\cdot)$ is strongly continuously differentiable and we denote the evolution family in Thm \ref{thm:generationL} by $\widetilde \T_l$, then the generator of the Howland semigroup of $\widetilde\T_l$ coincides with $\Gscr$ on the core $\dom{\Gscr_{0,n}}$. Then $\sigma\mapsto E^\sigma$ is the Howland semigroup of both $\T_l$ and $\widetilde \T_l$ and so $\widetilde \T_l=\T_l$ on $\Delta_{[a,b]}$.

Now the proof of \cite[Lemma 3.6]{ChenWeiss15} gives that $ \T_l(t,\tau)z=\lim_{n\to\infty} \T_n(t,\tau)z$ exists for all $z\in X$ and $(t,\tau)\in\Delta_{[a,b]}$, with uniform convergence over $(t,\tau)\in \Delta_{[a,b]}$ since here \linebreak $t-\tau\in[0,b-a]$. Then $ \T_l$ inherits properties 2)--4) in Def.\ \ref{def:evolfam}, with $[a,b]$ instead of $J$, from $\T_n$. Moreover,
$$
	\|\T_l(t,\tau)z\|=
	\lim_{n\to\infty} \|\T_n(t,\tau)z\| \leq
	Me^{\omega(t-\tau)}\|z\|,
$$
so that $\T_l(t,\tau)$ is bounded for all $t\in[a,b]$. This proves that $\T_l$ is an exponentially bounded evolution family on $X$ with time interval $[a,b]$. Moreover, $\T_l$ is seen to be an evolution family on $X$ with time interval $J$, by a simpler version of the extension outlined towards the end of the proof of Thm \ref{thm:generationL}.

\emph{Step 2: The unique solution of \eqref{eq:TlIntegEq} is $\T_l$.} This is a variation on the argument in \cite[Thm VI.9.19]{EnNa00}. Fix a compact $[a,b]\subset J$ arbitrarily. By step 1, $\Gscr_0$ defined as the closure $\Gscr_n$ of \eqref{eq:Gscrn} with $G_n=0$, and $\Gscr=\Gscr_0+M_G$, both generate Howland semigroups $E_\U$ and $E_{\T_l}$ with time interval $[a,b]$, corresponding to the locally exponentially bounded evolution families $\U$ and $\T_l$, respectively. By \cite[Cor.\ III.1.7]{EnNa00}, 
\begin{equation}\label{eq:SemiVC}
	E_{\T_l}^{t-\tau}f=E_\U^{t-\tau}f+\int_0^{t-\tau}E_\U^{t-\tau-s}M_GE_{\T_l}^sf\ud s
\end{equation}
holds in $C([a,b];X)$ for all $t-\tau\geq0$ and $f\in C([a,b];X)$. Letting $z\in X$ be arbitrary, picking some continuous $f$ with $f(s)=z$, and evaluating \eqref{eq:SemiVC} in $t$, we get the following from \eqref{eq:HowlInv}, where $\delta_t\in\Lscr\big(C([a,b];X);X\big)$ is point evaluation at $t\in[a,b]$:
\begin{equation}\label{eq:UTl}
\begin{aligned}
	\T_l(t,\tau)z-\U(t,\tau)z &=\\
	\delta_t\int_\tau^t E_\U^{t-r}M_GE_{\T_l}^{r-\tau}f\ud r &=\\
	\int_\tau^t \U(t,r)\,\delta_r(M_GE_{\T_l}^{r-\tau}f)\ud r &=\\
	\int_\tau^t \U(t,r)\,G(r)\, \T_l(r,\tau) z\ud r.&
\end{aligned}
\end{equation}
Thus $\T_l$ satisfies \eqref{eq:TlIntegEq} for all $(t,\tau)\in\Delta_{[a,b]}$, and since $[a,b]\subset J$ is arbitrary, \eqref{eq:TlIntegEq} holds for all $(t,\tau)\in\Delta_J$. Uniqueness follows by applying Gr\"onwall's theorem to $\phi(t):=\|\mathbb S(t,\tau)z-\T_l(t,\tau)z\|$, using that $\U(t,\cdot)$ and $G$ are uniformly bounded on every $[\tau,t]$.

\emph{Step 3: $x=\T_l(\cdot,\tau)x_\tau$ solves \eqref{eq:TlMild}.} Since $A_n(\cdot)$ generates $\T_n$ with time interval $J$, \eqref{eq:CauchyL} holds with $A_l(t)$ replaced by $A_n(t)$ and $x_\tau\in \dom{A_n(t)}$. Integrating, we get
\begin{equation}\label{eq:TnMild}
	\T_n(t,\tau)x_\tau-x_\tau=\int_\tau^t A_n(s)\,\T_n(s,\tau)x_\tau\ud s,
\end{equation}
as an equality in $X$, and hence in $X_{-1}$. We can replace $A_n(s)$ by its extension $A_{-1,n}(s)$ and compute the integral in $X_{-1}$ instead of in $X$.

For all $z\in X$ and $s\in[\tau,t]$,
$$
\begin{aligned}
	&\|A_{-1,l}(s)\,z\|_{-1}\leq \|\big((I-P(s)^{-1}A)^{-1}-I\big)\,z\| \\
	&\qquad+\|(I-P(s)^{-1}A)^{-1}P(s)^{-1}G(s)\,z\| \leq L\|z\|,
\end{aligned}
$$
by \cite[(2.7)]{SchWe10} and the uniform boundedness of $G(\cdot)$ on $[a,b]$. Local uniform boundedness of $A_{-1,l}(s)$ as an operator in $\Lscr(X;X_{-1})$ follows, and since $G_n$ has the same uniform bound as $G$, also $\|A_{-1,n}(s)\|_{\Lscr(X;X_{-1})}\leq L$ on compact subintervals of $J$. By a similar argument, for $f\in C([\tau,t];X)$ and $s\in[\tau,t]$, 
$$
\begin{aligned}
	\|A_{-1,n}(s)\,f(s)-A_{-1,l}(s)\,f(s)\|_{-1} &\leq \\
	N\,\|G_n(s)\,f(s)-G(s)\,f(s)\|&,
\end{aligned}
$$
and by \cite[(3.21)]{ChenWeiss15}, for all $f\in C([a,b];X)$:
\begin{equation}\label{eq:3.21}
	\lim_{n\to\infty}\sup_{s\in[a,b]} \|G_n(s)f(s)-G(s)f(s)\|=0.
\end{equation}
Gathering the above, we get for $x_\tau\in X$ and $s\in[\tau,t]$, that
$$
\begin{aligned}
	\|A_{-1,n}(s)\,\T_n(s,\tau)x_\tau-A_{-1,l}(s)\,\T_l(s,\tau)x_\tau\|_{-1} &\leq \\
	L\cdot \|\T_n(s,\tau)x_\tau-\T_l(s,\tau)x_\tau\|&\\
	+\|\big(A_{-1,n}(s)-A_{-1,l}(s)\big)\,\T_l(s,\tau)x_\tau\|&\to 0,
\end{aligned}
$$
uniformly in $s\in[\tau,t]$ as $n\to\infty$. Hence, letting $n\to\infty$ in \eqref{eq:TnMild}, we get \eqref{eq:TlMild} for $x_\tau\in\dom A$. From
$$
	\left\|\int_\tau^tA_{-1,l}(s)\,\T_l(s,\tau)x_\tau\ud s\right\|_{-1}\leq
	(t-\tau)\,LMe^{\omega (t-\tau)}\|x_\tau\|
$$
it follows that we can extend \eqref{eq:TlMild} by density to all of $X$.

\emph{Step 4: The Cauchy problem in $X_{-1}$.} Now let $z\in C([\tau,t];X)$ be any solution of \eqref{eq:TlMild}. Then $z(\tau)=x_\tau$ and it is clear that $z(\cdot)$ and its derivative $A_{-1,l}(\cdot)\,z(\cdot)$ are both in $C(J_\tau;X_{-1})$, and that $z(\cdot)$ is a solution of the Cauchy problem \eqref{eq:CauchyLextr}. Conversely, if $z\in C(J_\tau;X)\cap C^1(J_\tau;X_{-1})$ solves \eqref{eq:CauchyLextr}, then, by integration, $z$ also solves \eqref{eq:TlMild}.
\end{proof}

The following consequences of the above theorems are of particular interest when developing the theory for multiplicative perturbation from the right, the case corresponding to \eqref{eq:Sigmar}:

\begin{corollary}\label{cor:V}
The family
$$
	\widetilde A_n(t):=P(t)A+P(t)\,G_n(t)P(t)^{-1},\quad t\in J,
$$
generates an evolution family $\widetilde\T_n$ on $X$ with time interval $J$.

For all $(t,\tau)\in \Delta_J$ and $z\in X$, the strong limit
\begin{equation}\label{eq:PTtlP}
	\widetilde\T_l(t,\tau)z:=\lim_{n\to\infty}\widetilde\T_n(t,\tau)z,\quad z\in X,
\end{equation}
exists, with uniform convergence on $\Delta_{[a,b]}$ for compact intervals $[a,b]\subset J$, and
\begin{equation}\label{eq:Vdef}
	\V(t,\tau)z:=P(t)^{-1}\widetilde\T_l(t,\tau)P(\tau)z,\quad z\in X,
\end{equation}
is a locally exponentially bounded evolution family on $X$ with time interval $J$.

If $G(\cdot)z\in C^1([a,b];X)$ for all $z\in X$, then $\V$ is generated by $AP(t)+G(t)-P(t)^{-1}\dot P(t)$ with domain $P(t)^{-1}\dom A$, $t\in J$.
\end{corollary}

The first two claims follow from the uniform boundedness of $P(\cdot)$ and its inverse on compact subintervals of $J$, together with Thms \ref{thm:generationL} and \ref{thm:extensionL}. By the latter, if $G(\cdot)$ is strongly continuously differentiable, then the evolution family $(t,\tau)\mapsto P(t)\V(t,\tau)P(\tau)^{-1}$ is generated by $P(t)A+P(t)G(t)P(t)^{-1}$. Then the last assertion follows from \cite[Rem.\ 2.6]{SchWe10}. 

Next we need the extrapolation space of $AP(t)$. This space can be identified with $X_{-1}$, and $A_{-1,r}(t):=A_{-1}P(t)+G(t)$ is the unique extension to $\Lscr(X;X_{-1})$ of 
\begin{equation}\label{eq:ArDef}
\begin{aligned}
	A_r(t)&:=AP(t)+G(t), \\
	\dom{A_r(t)}&:=P(t)^{-1}\dom A;
\end{aligned}
\end{equation}
see the paragraph before \cite[Def.\ 2.5]{SchWe10}. 

\begin{theorem}\label{thm:generationR}
With $\V$ defined in \eqref{eq:Vdef}, there exists a unique evolution family $\T_r$ with time interval $J$, such that
\begin{equation}\label{eq:TrInt}
\begin{aligned}
	&\T_r(t,\tau)x_\tau=\V(t,\tau)x_\tau \\
	&\qquad +\int_\tau^t \V(t,s)P(s)^{-1}\dot P(s)\T_r(s,\tau)x_\tau\ud s,
\end{aligned}
\end{equation}
for all $x_\tau\in X$ and $(t,\tau)\in\Delta_J$. This $\T_r$ is also the unique solution of
\begin{equation}\label{eq:TrInt2}
\begin{aligned}
	&\T_r(t,\tau)x_\tau=\V(t,\tau)x_\tau \\
	&\qquad +\int_\tau^t \T_r(t,s)P(s)^{-1}\dot P(s)\V(s,\tau)x_\tau\ud s.
\end{aligned}
\end{equation}

For every $\tau\in J$ with $\tau<\sup\,J$ and $x_\tau\in X$, 
\begin{equation}\label{eq:xdef}
	x(t):=\T_r(t,\tau)x_\tau,\quad t\in J_\tau,
\end{equation}
is in $C^1(J_\tau,X_{-1})$ and it solves the Cauchy problem \eqref{eq:CauchyLextr}, with $A_{-1,l}$ replaced by $A_{-1,r}$, in $X_{-1}$.
\end{theorem}

Uniqueness in \eqref{eq:TrInt} and \eqref{eq:TrInt2} holds in the same sense as uniqueness in \eqref{eq:TlIntegEq}.

\begin{proof}
\emph{Step 1: Assertion one.} By temporarily restricting to compact subintervals of $J$, the proof of \cite[Prop.\ 2.7.a)]{SchWe10} gives the existence of the evolution family $\T_r$ that solves \eqref{eq:TrInt} and \eqref{eq:TrInt2} on all of $J$ as well as the uniqueness property for \eqref{eq:TrInt}. By \cite[Cor.\ 9.4]{CuPrBook}, \eqref{eq:TrInt2} also has at most one solution within the class described after the theorem statement.

\emph{Step 2: $\T_r$ is the limit of a sequence of evolution families whose adjoints have generators.} We modify the proof of \cite[Prop.\ 2.7.d)]{SchWe10} as follows, working on a compact subinterval $[a,b]\subset J$. On this interval $G(\cdot)$, $P(\cdot)$, $P(\cdot)^{-1}$ and $\dot P(\cdot)$ are uniformly bounded and the family
$$
\begin{aligned}
	\widehat A_n(t)&:=AP(t)+G_n(t)-P(t)^{-1}\dot P(t),\\
	\dom{\widehat A_n(t)}&:=P(t)^{-1}\dom A;
\end{aligned}
$$ 
generates an evolution family $\V_n$ with time interval $[a,b]$ by the proof sketch for Cor.\ \ref{cor:V}. For $t\in[a,b]$,
$$
	A_n^\dagger (t):= P(t)A^*+G_n(t)^*,\quad
	\dom{A_n^\dagger (t)}=\dom{A^*},
$$ 
generates a locally exponentially bounded \emph{backward} evolution family $\mathbb S_n$ with time interval $[a,b]$, see \cite[Def.\ 2.5]{SchWe10} and the discussion in step three of the proof of \cite[Prop.\ 5.1]{ChenWeiss15}, by the backward version of Thm \ref{thm:generationL}. We will prove that $\T_r(t,\tau)z=\lim_{n\to\infty} \mathbb S_n(t,\tau)^*z$ for all $z\in X$. 

First we calculate, for $w_n\in\dom {\widehat A_n}$ and $z\in\dom{A^*}$,
$$
\begin{aligned}
	\frac\partial{\partial s}\Ipdp{\mathbb S_n(t,s)^*\V_n(s,\tau)w_n}{z} &= \\
	\Ipdp{\widehat A_n(s)\V_n(s,\tau)w_n}{\mathbb S_n(t,s)z}\qquad& \\ -\Ipdp{\V_n(s,\tau)w_n}{A_n^\dagger(s)\mathbb S_n(t,s)z} &= \\
	-\Ipdp{P(s)^{-1}\dot P(s)\big)\V_n(s,\tau)w_n}{\mathbb S_n(t,s)z}&.
\end{aligned}
$$
Integrating this from $\tau$ to $t$ results in
\begin{equation}\label{eq:ConvWeakInt}
\begin{aligned}
	&\Ipdp{\mathbb S_n(t,\tau)^*w_n}{z} = \Ipdp{\V_n(t,\tau)w_n}{z} \\
	&\quad+\int_\tau^t\Ipdp{P(s)^{-1}\dot P(s)\V_n(s,\tau)w_n}{\mathbb S_n(t,s)z}\ud s.
\end{aligned}
\end{equation}
Assuming that $w_n\to w$ in $X$, we get $\V_n(s,\tau)w_n\to\V(s,\tau)w$ (since $\V_n$ have a uniform exponential bound on compact intervals which is independent of $n$ by the first assertion in Thm \ref{thm:extensionL}) and $\mathbb S_n(t,s)z\to \mathbb S(t,s)z$ for some $\mathbb S(t,s)z$, both uniformly on $s\in[\tau,t]$ by Cor.\ \ref{cor:V}. By the uniform boundedness of $P(s)^{-1}\dot P(s)$ on $[\tau,t]$, and the fact that $\mathbb S_n$ and $\V$ have some common exponential bound independent of $n$ and $s\in[\tau,t]$, we then get uniform convergence of the integrand, so that \eqref{eq:ConvWeakInt} tends to
$$
\begin{aligned}
	&\Ipdp{\mathbb S(t,\tau)^*w}{z} = \Ipdp{\V(t,\tau)w}{z} \\
	&\qquad+\int_\tau^t\Ipdp{\mathbb S(t,s)^*P(s)^{-1}\dot P(s)\V(s,\tau)w}{z}\ud s
\end{aligned}
$$
as $n\to\infty$, for $w\in X$ and $z\in\dom{A^*}$, and by denseness also for all $z\in X$. From the uniqueness of solution of \eqref{eq:TrInt2}, we get $\mathbb S(t,\tau)^*=\T_r(t,\tau)$.

\emph{Step 3: Assertion two.} Now we establish that $\T_r(\cdot,\tau)x_0$ solves the Cauchy problem \eqref{eq:CauchyLextr}, with $l$ replaced by $r$. We have $\frac\partial{\partial s}\mathbb S_n(s,\tau)z_0=\mathbb S_n(s,\tau)A_n^\dagger(s)z_0$ for all $z_0\in\dom{A^*}$ and $\tau\in[a,b]$. Integrating from $\tau$ to $t$, we get
$$
	\mathbb S_n(t,\tau)z_0-z_0=
	\int_\tau^t \mathbb S_n(s,\tau)A_n^\dagger(s)z_0\ud s,
$$
and for an arbitrary $x_\tau\in X$, we have
$$
\begin{aligned}
	&\Ipdp{\mathbb S_n(t,\tau)^*x_\tau}{z_0}-\Ipdp{x_\tau}{z_0} = \\
	&\quad\int_\tau^t \Ipdp{\big(A_{-1}P(s)+G_n(s)\big)\,\mathbb S_n(s,\tau)^*x_\tau}{z_0}_{X_{-1},\dom{A^*}}\ud s.
\end{aligned}
$$
Letting $n\to\infty$, strong convergence of $G_n$ to $G$ in $C([a,b];X)$ gives
$$
	\T_r(t,\tau)x_\tau-x_\tau = 
	\int_\tau^t A_{-1,r}(s)\,\T_r(s,\tau)x_\tau\ud s
$$
in $X_{-1}$. Proceeding as in step 3 of the proof of Thm \ref{thm:extensionL}, we get assertion three.
\end{proof}

We can say more about $\T_r$ if $P(\cdot)$ and $G(\cdot)$ are smoother:

\begin{proposition}\label{prop:GenRsmooth}
Assume that $P(\cdot)z\in C^2(J;X)$ and $G(\cdot)z\in C^1(J;X)$ for all $z\in X$.

The family $A_r(\cdot)$ in \eqref{eq:ArDef} generates $\T_r$ in Thm \ref{thm:generationR}, i.e., 
\begin{equation}\label{eq:TrDomMap}
	\T_r(t,\tau)\,\dom{A_r(\tau)}\subset\dom{A_r(t)}
\end{equation}
for $(t,\tau)\in\Delta_J$, and $x$ in \eqref{eq:xdef} is a solution in $C^1(J_\tau,X)$ of the Cauchy problem in $X$, instead of only in $X_{-1}$. The map
$$
	\Delta_J\ni (t,\tau)\mapsto AP(t)\,\T_r(t,\tau)\,\big(I-AP(\tau)\big)^{-1}z
$$
is continuous in $X$, for all $z\in X$.
\end{proposition}

\begin{proof}
Due to Thm \ref{thm:generationL}, the family
$$
	A_2(t):=P(t)A+P(t)\,G(t)P(t)^{-1}+\dot P(t)P(t)^{-1},~ t\in J,
$$ 
with $\dom{A_2(t)}=\dom A$, generates an evolution family $\mathbb S$ such that $(t,\tau)\mapsto A_2(t)\,\mathbb S(t,\tau)x_\tau$ is in $C(\Delta_J;X)$ for all $x_\tau\in \dom A$, since $P(t)\,G(t)P(t)^{-1}+\dot P(t)P(t)^{-1}$ is strongly continuously differentiable. Then $A_r(t)$ generates the evolution family $(t,\tau)\mapsto P(t)^{-1}\mathbb S(t,\tau)P(\tau)$ by the proof sketch of Cor.\ \ref{cor:V}. The family
$P(t)A+P(t)G(t)P(t)^{-1}$ also generates a locally exponentially bounded evolution family $\widetilde\T_l$, and by Cor.\ \ref{cor:V}, $P(t)^{-1}\widetilde\T_l(t,\tau)P(\tau)=\V(t,\tau)$. By Thm \ref{thm:generationL}, for $z\in\dom A$, $\mathbb S(t,\tau)z\in\dom A$ and
$$
	\frac{\partial}{\partial s} \widetilde\T_l(t,s)\mathbb S(s,\tau)z =
	\widetilde\T_l(t,s)\dot P(s)P(s)^{-1}\mathbb S(s,\tau)z.
$$
Proceeding as in the proof of \cite [Prop.2.8.a)]{SchWe10}, we get $P(t)^{-1}\mathbb S(t,\tau)P(\tau)=\T_r(t,\tau)$ and the claimed continuity.
\end{proof}

\section{Well-posedness of $\Sigma_l$ and $\Sigma_r$}\label{sec:WP}

We use the Lax-Phillips evolution families associated to  \eqref{eq:Sigmal} and \eqref{eq:Sigmar} in order to prove the well-posedness of these time-varying linear systems and their passivity property.

\subsection{Multiplicative perturbation from the left}

Let $\sbm{A&B\\\overline C&D}$ be a time-invariant, passive well-posed system, with Lax-Phillips semigroup $\Tfrak^i$ on $\mathcal H$, generated by the maximal dissipative operator $\Afrak$ in \eqref{eq:LPsgGen}. By a \emph{classical trajectory} of the time-varying system
\begin{equation}\label{eq:AlSys}
	\bbm{P(t)\,\dot x(t)\\y(t)}=\bbm{A_{-1}+P(t)\,G(t)&B\\\overline C&D}\bbm{x(t)\\u(t)}
\end{equation}
on $J_\tau$, $\tau\in J$ with $\tau<\sup\, (J)$, we mean a triple $(u,x,y)\in C(J_\tau;U)\times C^1(J_\tau;X)\times C(J_\tau;Y)$, such that \eqref{eq:AlSys} holds for all $t\in J_\tau$.

For $t\in J$, define the following operators in $\Lscr(\Hscr)$:
$$
	\Pfrak(t):=\bbm {I&0&0 \\ 0&P(t)&0 \\ 0&0&I},\quad\text{and}\quad
	\Gfrak(t):=\bbm {0&0&0 \\ 0&G(t)&0 \\ 0&0&0},
$$
so that $\Afrak_l(t):=\Pfrak(t)^{-1}\Afrak+\Gfrak(t)$ equals
$$
	\Afrak_l(t)\bbm{y\\x_0\\u} =\bbm{y'\\P(t)^{-1}\big(A_{-1}x_0+Bu(0)\big)+G(t)x_0\\u'},
$$
\begin{equation}\label{eq:LPlGen}
	\qquad\text{with}\quad \dom{\Afrak_l(t)}=\dom\Afrak,\quad t\in J.
\end{equation}
Thm \ref{thm:extensionL} implies that the family $\Afrak_l(t)$ generates an evolution family $\Tfrak_l$ on $\Hscr$ with time interval $J$ if $G(\cdot)$ is strongly continuously differentiable, and we will prove that $\Tfrak_l$ is the Lax-Phillips evolution family  of a (unique, time-varying) well-posed system $\Sigma_l$, whose classical trajectories with smooth data are determined by \eqref{eq:Sigmal}.

\begin{theorem}\label{thm:Sigmal}
Let $\tau\in J$ be such that $\tau<\sup J$. There is a well-posed system $\Sigma_l=\sbm{\T_l&\Phi_l\\\Psi_l&\F_l}$ with Lax-Phillips evolution family $\Tfrak_l$ and time interval $J$, such that the evolution family $\T_l$ of $\Sigma_l$ has the properties asserted in Thm \ref{thm:extensionL}, and moreover:
\begin{enumerate}
\item Let $(u,x,y)$ be a classical trajectory of \eqref{eq:AlSys} on $J_\tau$. Then $x\in C(J_\tau;Z)$, where $Z$ is the solution space in \eqref{eq:solspace}. Moreover, if $(x(\tau),u)$ is in
$$
\begin{aligned}
	V(\tau)&:=\bigg\{\bbm{x_\tau\\u}\in \bbm{X\\ H^1(J;U)}\biggmid \\
	&\qquad\qquad\qquad\qquad\quad A_{-1}x_\tau+Bu(\tau)\in X\bigg\},
\end{aligned}
$$
then $(u,x,y)$ is a trajectory of $\Sigma_l$ on $J_\tau$ with $x(\tau)=x_\tau$:
\begin{equation}\label{eq:SigmalTraj}
\begin{aligned}
	x(t)&=\T_l(t,\tau)x_\tau+\Phi_l(t,\tau)u,\\
	\proj_{[\tau,t]}y&=\Psi_l(t,\tau) x_\tau+\F_l(t,\tau)u,\quad t\in J_\tau.
\end{aligned}
\end{equation}

\item For $x_\tau\in X$ and $u\in L^2(J_\tau;X)$, the function
$$
	x(t):=\T_l(t,\tau)x_\tau+\Phi_l(t,\tau)u,\qquad t\in J_\tau,
$$
is in $H^1_{loc}(J_\tau;X_{-1})$ and it satisfies 
$$
	\dot x(t)=A_{-1,l}(t)\,x(t)+Bu(t)
$$
in $X_{-1}$ for almost all $t\in J_\tau$, where $A_{-1,l}$ is in \eqref{eq:Amin1lexp}.

\item Every trajectory $(u,x,y)$ of $\Sigma_l$ on $J_\tau$ satisfies the energy inequality \eqref{eq:SigmalEnergy}, for all $t\in J_\tau$, and it is uniquely determined by $x(\tau)$ and $u$.

\item[4)] Every classical trajectory of \eqref{eq:AlSys} satisfies
\begin{equation}\label{eq:SigmalPower}
\begin{aligned}
	&\ddt\Ipdp{P(t)x(t)}{x(t)}\leq \Ipdp{\dot P(t)x(t)}{x(t)}\\
		&\qquad +\|u(t)\|^2-\|y(t)\|^2\\
		&\qquad +2\re\Ipdp{P(t)x(t)}{G(t)x(t)},\quad t\in J_\tau.
\end{aligned}
\end{equation}
\end{enumerate}

If additionally $G(\cdot)z\in C^1(J;X)$ for all $z\in X$, then $\T_l$ has the properties asserted in Thm \ref{thm:generationL}, and further:
\begin{enumerate}
\item[5)] The Lax-Phillips evolution family $\Tfrak_l$ of $\Sigma_l$ is generated by $\Afrak_l(t)$, $t\in J$, in the sense of Def.\ \ref{def:evolfam}.

\item[6)] For every $(x_\tau,u)\in V(\tau)$ there is a (unique) classical trajectory $(u,x,y)$ of \eqref{eq:AlSys} on $J_\tau$ with $x(\tau)=x_\tau$. The output of this trajectory satisfies $y\in H_{loc}^1(J_\tau;Y)$.
\end{enumerate}

If $\Sigma_i$ is energy preserving then the inequality holds with equality in \eqref{eq:SigmalEnergy} and \eqref{eq:SigmalPower}.
\end{theorem}

Hence, solutions $x\in C^1(J_\tau;X)$ of \eqref{eq:TlMild}, or equivalently, of \eqref{eq:CauchyLextr}, for which $x_\tau\in \dom A$, are unique. Indeed, they satisfy $x\in C(J_\tau;\dom A)\subset C(J_\tau;Z)$, see the penultimate paragraph of \S\ref{sec:prel}, and then $(0,x,y)$ is a classical trajectory of \eqref{eq:AlSys} with $y(t):=\overline C x(t)$, $t\in J_\tau$, and $G(\cdot):=0$.

\begin{proof}
\emph{Step 1: Power and energy balance for classical trajectories. Uniqueness of trajectories.} First let $(u,x,y)$ be an arbitrary classical trajectory of \eqref{eq:AlSys}; then
$$
	\bbm{P(t)\dot x(t)-P(t)G(t)x(t)\\y(t)}=
	\bbm{A_{-1}&B\\\overline C&D}\bbm{x(t)\\u(t)},
$$
and \eqref{eq:passpow} gives
$$
\begin{aligned}
	\|u(t)\|^2-\|y(t)\|^2&\geq\\ 
	2\re \Ipdp{P(t)\dot x(t)-P(t)G(t)x(t)}{x(t)} &= \\
	2\re \Ipdp{P(t)\dot x(t)}{x(t)} 
		-2\re \Ipdp{G(t)x(t)}{P(t)x(t)}&,
\end{aligned}
$$
with equality if $\Sigma_i$ is energy preserving, and observing that
$$
\begin{aligned}
	\ddt\Ipdp{P(t)x(t)}{x(t)}&= \\
	\Ipdp{\dot P(t)x(t)}{x(t)}+2\re\Ipdp{P(t)\dot x(t)}{x(t)}&,
\end{aligned}
$$
we get \eqref{eq:SigmalPower}. Integrating \eqref{eq:SigmalPower} from $\tau$ to $t\geq\tau$, we get \eqref{eq:SigmalEnergy}. 

In every linear set of triples $(u,x,y)$ that satisfy \eqref{eq:SigmalEnergy}, there is at most one $(u,x,y)$ with a particular choice of $x_\tau:=x(\tau)$ and $u$; indeed, setting $\phi(t):=\Ipdp{P(t)x(t)}{x(t)}+\int_\tau^t\|y(s)\|^2\ud s$, we get from \eqref{eq:SigmalEnergy} that
$$
\begin{aligned}
	\phi(t)&\leq \alpha(t) +M\int_\tau^t \phi(s)\ud s,
	\qquad\text{where} \\
	\alpha(t)&:=\Ipdp{P(\tau)x(\tau)}{x(\tau)}+\int_\tau^t \|u(s)\|^2\ud s,
\end{aligned}$$
is non-decreasing in $t$ and
$$ 	
\begin{aligned}
	M:=\sup_{s\in[\tau,t]} &\big( \|P(s)^{-1/2}\dot P(s)P(s)^{-1/2}\| \\
	&\quad +2 \,\|P(s)^{1/2}G(s)P(s)^{-1/2}\|\big),
\end{aligned}
$$
and Gr\"onwall's inequality then gives that
$$
	\phi(t)\leq\alpha(t)\,e^{M(t-\tau)},
$$
where $\alpha(t)$ is identically zero if $x(\tau)=0$ and $u=0$.

\emph{Step 2: If $G(\cdot)$ is strongly continuously differentiable.} Following steps one and two of the proof of \cite[Thm 4.1]{SchWe10} and using Thm \ref{thm:generationL}, we get the following: $\Afrak_l(\cdot)$ generates an evolution family $\Tfrak_l$ on $\Hscr$ with time interval $J$, and for all $(x_\tau,u)\in V(\tau)$ there exist $\widetilde y\in H^1(\rminus;Y)$ and $\widetilde u\in H^1((\tau,\infty);U)$ such that $\widetilde y(0)=\overline Cx_\tau+Du(\tau)$ and $\proj_{J_\tau}\widetilde u=u$. Then $(\widetilde y,x_\tau,\shift_\tau\widetilde u)\in \dom{\Afrak}$ and 

\begin{equation}\label{eq:LPlInterp}
	\bbm{y(t)\\x(t)\\u(t)}:= \bbm{\delta_0&0&0\\0&I&0\\0&0&\delta_0} 
	\Tfrak_l(t,\tau) \bbm{\widetilde y\\x_\tau\\\shift_\tau \widetilde u}
\end{equation}
defines a classical trajectory $(u,x,y)$ of \eqref{eq:AlSys} on $J_\tau$, with initial state $x(\tau)=x_\tau$ and output satisfying $\proj_{[\tau,t]}y\in H^1([\tau,t];Y)$ for $t\in J_\tau$; here $\delta_s$ is point evaluation at $s$.

Item 5) is proved in step five of the proof of \cite[Thm 4.1]{SchWe10}. Hence, items 4)--6) are proved and 1) holds for trajectories as in 6) if $G(\cdot)z\in C^1(J;X)$.

\emph{Step 3: Proving 1).} Now we drop the additional smoothness assumption on $G$, and we roughly follow the proof of \cite[Thm 5.3(b--c)]{ChenWeiss15}. Let $\Gfrak_n$ be defined as $\Gfrak$, but using the averaged function $G_n$ in \eqref{eq:PnDef}, instead of $G$, so that
$$
	\Afrak_{n}(t):=\Pfrak(t)^{-1}\Afrak+\Gfrak_n(t),
	\quad \dom{\Afrak_{n}(t)}=\dom\Afrak,
$$
generates the locally exponentially bounded Lax-Phillips evolution family $\Tfrak_{n}$ of a well-posed system $\Sigma_n$ with time interval $J$. By Thm \ref{thm:extensionL}, $\Tfrak_n(t,\tau)w$ tends uniformly to some $\Tfrak_l(t,\tau)w$ in $\Hscr$ on $\Delta_{[a,b]}$ for compact $[a,b]\subset J$ and all $w\in\Hscr$. Then $\Tfrak_l$ has the structure \eqref{eq:LPevol} and the operator families in $\Tfrak_l$ inherit causality \eqref{eq:WPcausal} from those in $\Tfrak_n$, so that $\Tfrak_l$ is the Lax-Phillips evolution family of a well-posed system $\Sigma_l$ by Thm \ref{thm:LPsys}.

For $(u,x,y)$ an arbitrary classical trajectory of \eqref{eq:AlSys},
$$
\begin{aligned}
	x(t)&=(I-A_{-1})^{-1}\big( (I+P(t)G(t))x(t)-P(t)\dot x(t) \big) \\
	&\qquad +(I-A_{-1})^{-1}Bu(t),
\end{aligned}
$$
where the first term is in $C(J_\tau;\dom A)$ which is contained in $C(J_\tau;Z)$ due to the continuity of the embedding $\dom A\to Z$, and the second term is in $C(J_\tau;Z)$ since $(I-A_{-1})^{-1}B\in \Lscr(U,Z)$; see the discussion around \eqref{eq:Znorm}.

Now moreover assume that $(x(\tau),u)\in V(\tau)$. Further let $(u,x_n,y_n)$ be the unique classical trajectory of \eqref{eq:AlSys} with $G$ replaced by $G_n$, such that $x_n(\tau)=x_\tau:=x(\tau)$; then
\begin{equation}\label{eq:trajn}
\begin{aligned}
	x_n(t)&=\T_n(t,\tau)x_\tau+\Phi_n(t,\tau)u,\\
	\proj_{[\tau,b]}\,y_n&=\Psi_n(b,\tau) x_\tau+\F_n(b,\tau)u,\quad t,b\in J_\tau.
\end{aligned}
\end{equation}
Letting $n\to\infty$, we get from the construction of $\Sigma_l$ that $x_n\to x_\infty:=\T_l(\cdot,\tau)x_\tau+\Phi_l(\cdot,\tau)u$ uniformly on $[\tau,b]\subset J_\tau$ and $\proj_{[\tau,b]}y_n\to y_\infty:=\Psi_l(b,\tau) x_\tau+\F_l(b,\tau)u$ in $L^2([\tau,b];U)$, so that $(u,x_\infty,y_\infty)$ is a trajectory of $\Sigma_l$ on $[\tau,b]$ with $x_\infty(\tau)=x_\tau$. We next prove that $x_\infty=x$ and $y_\infty=\proj_{[\tau,b]}y$. 

As $\Pfrak(t)^{-1}\Afrak+\Gfrak_n(t)$ generates $\Tfrak_{n}$ and $(x(\tau),u)\in V(\tau)$, \eqref{eq:LPevol}, \eqref{eq:LPlGen} and \eqref{eq:LPsgGen} give
\begin{equation}\label{eq:trajnLHS}
\begin{aligned}
	\dot x_n(t)&=P(t)_{-1}^{-1}\big(A_{-1}x_n(t)+Bu(t)\big)+G_n(t)x_n(t),\\
	y_n(t)&=\bbm{\delta_0&0&0}\Tfrak_{n}(t,\tau)\bbm{\widetilde y\\x_\tau\\\shift_\tau u}\\
	&=\overline Cx_n(t)+Du(t),
\end{aligned}
\end{equation}
for some $(\widetilde y,x_\tau,\shift_\tau u)\in\dom\Afrak$ and $\proj_{[\tau,t]}(\widetilde u-u)=0$.

Then $w_n(t):=x_n(t)-x(t)\in C^1(J_\tau;X)$ solves the inhomogeneous Cauchy problem
$$
	\dot w_n(t)=A_{-1,n}(t)\, w_n(t)+v_n(t),\quad t\in J_\tau,\quad w_n(\tau)=0,
$$
with $v_n(t):=\big(G_n(t)-G(t)\big)x(t)$ in $C(J_\tau;X)$. Then, in fact $A_{-1,n}(t)\, w_n(t)=\dot w_n(t)-v_n(t)\in C(J_\tau;X)$, so that 
$w_n\in C(J_\tau;X_1)$ and $A_{-1,n}(t)\, w_n(t)=A_{n}(t)\, w_n(t)$, $t\in J_\tau$. By \cite[Thms 5.4.8 and 5.4.2]{PazyBook},
$$
	w_n(t)=\int_\tau^t\T_n(t,s)\,v_n(s)\ud s,
$$
and the local exponential boundedness of $\T_n$, independent of $n$, implies that 
$$
\begin{aligned}
	\|w_n(t)\|&\leq \int_\tau^t Me^{\omega (t-s)}\|v_n(s)\|\ud s\\
	&\leq (b-\tau)Me^{\omega (b-\tau)}\|\proj_{[\tau,b]}v_n\|_\infty
		\to0,
\end{aligned}
$$
for $t\in[\tau,b]$, since w.l.o.g. $\omega\geq0$, and $\|\proj_{[\tau,b]}v_n\|_\infty\to0$ for all $[\tau,b]\subset J_\tau$ by \eqref{eq:3.21}. This gives $x_\infty=x$.

Since $\sbm{A_{-1}&B\\\overline C&D}$ is passive, in analogy to the proof of \eqref{eq:SigmalPower},
$$
\begin{aligned}
	&\|y_n(t)-y(t)\|^2 =\|Cw_n(t)\|^2\leq -2\re\Ipdp{Aw_n(t)}{w_n(t)}\\
	&\quad = -2\re\Ipdp{\dot w_n(t)-G_n(t)x_n(t)+G(t)x(t)}{P(t)w_n(t)} \\
	&\quad =-\ddt\Ipdp{P(t)w_n(t)}{w_n(t)}
		+\Ipdp{\dot P(t)w_n(t)}{w_n(t)} \\
	&\qquad+2\re\Ipdp{G_n(t)x_n(t)-G(t)x(t)}{P(t)w_n(t)}.
\end{aligned}
$$
Integrating from $\tau$ to $b$, we get
$$
	\|\proj_{[\tau,b]}(y_n-y)\|_{L^2(J_\tau;Y)}^2\leq
	(b-\tau)K_n\cdot \| \proj_{[\tau,t]}w_n \|_\infty,
$$
$$
\begin{aligned}
	K_n&:=\max_{t\in[\tau,b]}\|\dot P(t)\|\cdot\|w_n(t)\|\\
	&\quad +2\max_{t\in[\tau,b]}\|P(t)\|\cdot\|G_n(t)x_n(t)-G(t)x(t)\|,
\end{aligned}
$$
which has a bound independent of $n$, since $x_n$ tends uniformly to $x$, and then $y_\infty=\lim_{n\to\infty}\proj_{[\tau,b]}y_n=\proj_{[\tau,b]}y$. Since $b\in J_\tau$ was arbitrary, we have proved item 1).

\emph{Step 4: $\T_l$ has the properties in Thm \ref{thm:extensionL} and possibly those in Thm \ref{thm:generationL}.} When $G$ is strongly continuously differentiable, $A_l(t)$ generates an evolution semigroup $\widetilde \T_l$ with time interval $J$, by Thm \ref{thm:generationL}. In order to prove that $\widetilde \T_l=\T_l$, fix $\tau\in J$, $\tau<\sup J$, and $x_\tau\in\dom A$ arbitrarily. Then $\widetilde x(t):=\widetilde \T_l(t,\tau)x_\tau$, $t\in J_\tau$, solves \eqref{eq:CauchyL}. By defining $\widetilde y(t):=\overline C\,\widetilde x(t)$, $t\in J_\tau$, we get that $(0,\widetilde x,\widetilde y)$ is a classical trajectory of \eqref{eq:AlSys} on $J_\tau$ with $\widetilde x(\tau)=x_\tau$. Since $(x_\tau,0)\in V(\tau)$, there is only one such trajectory by step 2, and $\widetilde x(t)=\T_l(t,\tau)x_\tau$ for all $t\in J_\tau$. This proves that $\widetilde\T_l(t,\tau)$ and $\T_l(t,\tau)$ coincide on the dense subspace $\dom A$ of $X$, and by the boundedness of these operators, $\widetilde\T_l(t,\tau)=\T_l(t,\tau)$ on all of $X$, for all $(t, \tau)\in \Delta_J$.

Now, even if $G$ is only strongly continuous, $\Pfrak(\cdot)^{-1}\Afrak$ generates the Lax-Phillips evolution family of some well-posed system $\sbm{\U&\star\\\star&\star}$ by step 2, and by what we just proved, $\U$ is generated by $P(\cdot)^{-1}A$ and it has all the properties asserted in \ref{thm:generationL}. By Thm \ref{thm:extensionL}, the Lax-Phillips evolution family $\Tfrak_l$ constructed in step 2 satisfies the equation
$$
\begin{aligned}
	&\bbm{0&I&0}\Tfrak_l(t,\tau)\bbm{0\\z\\0} =
	\bbm{0&I&0}\Ufrak(t,\tau)\bbm{0\\z\\0}\\
	&\qquad+\int_\tau^t \bbm{0&I&0}\Ufrak(t,s)\Gfrak(s)\Tfrak_l(s,\tau)
	\bbm{0\\z\\0}\ud s,
\end{aligned}
$$
and by \eqref{eq:LPevol} and the definition of $\Gfrak$, this equals \eqref{eq:TlIntegEq}. By uniqueness in \eqref{eq:TlIntegEq}, Since $\T_l$ inherits local exponential boundedness from $\Tfrak_l$, $\T_l$ is equal to the evolution family in Thm \ref{thm:extensionL}; hence it has all the asserted properties.

\emph{Step 5: Proving 2).} Define $x(t)$ as in 2), but with $(x_\tau,u)\in V(\tau)$ such that $\supp u\subset J_\tau$ w.l.o.g, and \emph{define} $x_n$ by \eqref{eq:trajn}. Then the calculations in step 3 give that $x(t)=\lim_{n\to\infty} x_n(t)$ in $X$, uniformly on $[\tau,b]\subset J$. By \eqref{eq:trajnLHS}, $x_n(t)-x_\tau$ equals
$$
	\int_\tau^t P(s)_{-1}^{-1}\big(A_{-1}x_n(s)+Bu(s)\big)+G_n(s)x_n(s)\ud s
$$ 
in $X_{-1}$, and letting $n\to\infty$, \eqref{eq:3.21} gives that $G_nx_n\to Gx$ uniformly on $[\tau,t]$ too, so that
\begin{equation}\label{eq:CauchIntl}
	x(t)-x_\tau=\int_\tau^tA_{-1,l}(s)\,x(s)+Bu(s)\ud s.
\end{equation}

For arbitrary $x_\tau\in X$ and $u\in L^2(J_\tau;U)$, $x\in C(J_\tau;X)$ which implies that $[\tau,b]\ni t\mapsto A_{-1,l}(t)\,x(t)+Bu(t)$ is in $L^2([\tau,b];X_{-1})$ for all compact $[\tau,b]\subset J$. Approximating $(x_\tau,u)$ by elements in $V(\tau)$ and using that convergence in $L^2$ implies convergence in $L^1$ on compact intervals in \eqref{eq:CauchIntl}, together with local uniform convergence of state trajectories due to the local uniform boundedness of $\Tfrak_l$ and $\Phi_l$, we obtain that \eqref{eq:CauchIntl} still holds. Hence, $x\in C(J_\tau;X)\subset L^2_{loc}(J_\tau;X_{-1})$ is a primitive of the $L^2_{loc}(J_\tau;X_{-1})$ function $\dot x(t)= A_{-1,l}(t)\,x(t)+Bu(t)$.

\emph{Step 6: Proving 3).} Let $(x_\tau,u)\in V(\tau)$ be arbitrary. By step 2, the functions $x_n\in C(J_\tau;X)$ and $y_n\in L^2_{loc}(J_\tau;Y)$ determined by \eqref{eq:trajn} 
satisfy \eqref{eq:SigmalEnergy} with $G_n$ instead of $G$. Letting $n\to\infty$, we get from step 4 that $(u,x,y)$ determined by \eqref{eq:SigmalTraj} satisfy \eqref{eq:SigmalEnergy}, due to the uniform convergence of $x_n$ to $x$ and of $G_nx_n$ to $Gx$ on $[\tau,t]$. Approximating as in step 4, we  extend \eqref{eq:SigmalEnergy} to arbitrary $(x_\tau,u)\in X\times L^2(J_\tau;U)$.
\end{proof}

Under stronger regularity assumptions on $P(\cdot)$ and $G(\cdot)$, we have the following representation formulas:

\begin{proposition}\label{prop:repL}
Assume that $P(\cdot)z\in C^2(J;X)$ and that $G(\cdot)z,\,G(\cdot)^*z\in C^1(J;X)$ for all $z\in X$. Let $\tau\in J$ with $\tau<\sup J$ and $t\in J_\tau$. 

The operators $\T_l(t,\tau)$ have unique extensions to $\T_{-1,l}(t,\tau)\in\Lscr(X_{-1,r}^\tau;X_{-1,r}^t)$ which are locally uniformly bounded. For all $x_0\in X$, $\T_l(t,\cdot)x_0$ is in $C^1((\inf J,t];X_{-1})$, 
$$
	\frac{\partial}{\partial\tau}\T_l(t,\tau)x_0=-\T_{-1,l}(t,\tau)A_{-1,l}(\tau)x_0.
$$

For all $x_\tau\in \dom A$,
$$
	\Psi_l(t,\tau)x_\tau=s\mapsto C\,\T_l(s,\tau)x_\tau,
		\quad \tau\leq s\leq t,
$$
and for all $(0,u)\in V(\tau)$,
$$
\begin{aligned}
	\Phi_l(t,\tau)u &= \int_\tau^t \T_{-1,l}(t,s)P(s)^{-1}_{-1}Bu(s)\ud s \\
	\big(\F_l(t,\tau)u\big)(s) &= \overline C\int_\tau^s \T_{-1,l}(s,\sigma)P(\sigma)^{-1}_{-1}Bu(\sigma)\ud \sigma \\
	&\qquad + Du(s),\qquad s\in[\tau,t],
\end{aligned}
$$ 
where the integrals are computed in $X_{-1}$. The representation formula for $\Phi_l$ in fact holds for all $u\in L_{loc}^2(J;U)$.
\end{proposition}

\begin{proof}
The first assertion follows from the proof of \cite[Prop.\ 2.8(b)]{SchWe10} with some minor modifications: $P(\cdot)^{-1}z\in C^2(J;X)$ for all $z\in X$ and then the backward analogue of  Prop.\ \ref{prop:GenRsmooth} gives that $A_{-1,l}(t)^*=A^*P(t)^{-1}+G(t)^*$ with $\dom{A_{-1,l}(t)^*}=P(t)\,\dom{A^*}$ generates the backward evolution family $\T_l^*$ with time interval $J$, such that
$$
	\Delta_J\ni (t,\tau)\mapsto A^*P(\tau)^{-1}\,\T_l(t,\tau)^*\,\big(I-A^*P(t)^{-1}\big)^{-1}z
$$
is continuous in $X$, for all $z\in X$. Then
$$
	T(t,\tau):=\big(P(\tau)-A^*\big) P(\tau)^{-1}\T_l(t,\tau)^*\big(I-A^*P(t)^{-1}\big)^{-1}
$$
is strongly continuous on $\Delta_J$ and by the uniform boundedness principle, $\|T(t,\tau)\|\leq K_{[a,b]}$ on $\Delta_{[a,b]}$ for some $K_{[a,b]}$ depending on the compact $[a,b]\subset J$. Then
$$
\begin{aligned}
	\|\T_{-1,l}(t,\tau)\|_{\Lscr(X_{-1,l}^\tau;X_{-1,l}^t)} &=
		\| \T_{l}(t,\tau)^*\|_{\Lscr(X_{1,l}^{d,t};X_{1,l}^{d,\tau})} \\
	&\leq K_{[a,b]},\quad (t,\tau)\in \Delta_{[a,b]},
\end{aligned}
$$
with $\T_{-1,l}(t,\tau)$ the dual of $\T_{l}(t,\tau)^*$. The first assertion now follows from the density of $X$ in $X_{-1,l}^\tau$. Then the second assertion follows as in the proof of \cite[Prop.\ 4.2b)]{SchWe10}.

Let $x_\tau\in\dom A$; then $(x_\tau,0)\in V(\tau)$ and by Thm \ref{thm:Sigmal}, there is a unique classical trajectory $(0,x,y)$ of \eqref{eq:AlSys} on $J_\tau$ with $x(\tau)=x_\tau$. This trajectory moreover satisfies 
$$
	x(t)=\T_l(t,\tau)x_\tau\qquad\text{and}\qquad \proj_{[\tau,t]}y=\Psi_l(t,\tau)x_\tau,
$$
where the latter is in $C([\tau,t];Y)$. Then, for $s\in[\tau,t]$,
$$
	(\Psi_l(t,\tau)x_\tau)(s)=y(s)=\overline Cx(s)=\overline C\,\T_l(s,\tau)x_\tau.
$$
Recalling that $C=\overline C\big|_{\dom A}$ and that $\proj_{[\tau,t]}y=\Psi_l(t,\tau)x_\tau$ by Thm \ref{thm:Sigmal}, we get the representation formula for $\Psi_l$.

For the other assertions, by using Thm \ref{thm:Sigmal} in the proof of \cite[Prop.\ 4.2.b)]{SchWe10}, we get for $(0,u)\in V(\tau)$, in $X_{-1}$,
$$
	\frac{\partial}{\partial s} \T_{l}(t,s)\,\Phi_l(s,\tau)u = 
	\T_{-1,l}(t,s)P(s)^{-1}_{-1}Bu(s),
$$
and integration now gives the representation formula for $\Phi_l(t,\tau)u$ for $(0,u)\in V(\tau)$. By Thm \ref{thm:Sigmal}, $x:=\Phi_l(\cdot,\tau)u\in C(J_\tau;Z)$ is the state trajectory of \eqref{eq:AlSys} corresponding to $(0,u)\in V(\tau)$, and the output $y\in C(J_\tau;Y)$ of \eqref{eq:AlSys} is
$$
	\big(\F_l(t,\tau)u\big)(s)=y(s)=\overline C\,\Phi_l(s,\tau)u+Du(s),
$$
i.e., the representation formula for $\F_l(t,\tau)u$ is correct. 

By the boundedness of $\Phi_l(t,\tau)$, density, the uniform boundedness on $[\tau,t]$ of $\T_{-1,l}(t,s)P(s)^{-1}_{-1}B$ from $U$ into $X_{-1}$, and the fact that convergence in $L^2$ implies convergence in $L^1$ on $[\tau,t]$, the representation formula for $\Phi_l$ is correct even for all $u\in L^2(J;U)$. Observing that the integration in the representation formula happens over the compact interval $[\tau,t]$, we get the representation formula for all $u\in L_{loc}^2(J;U)$.
\end{proof}

\subsection{Multiplicative perturbation from the right}

The two preceding results have counterparts in the case of a multiplicative perturbation from the right. In the remainder of this section, we are concerned with the system
\begin{equation}\label{eq:ArSys}
	\bbm{\dot x(t)\\y(t)}=\bbm{A_{-1}P(t)+G(t)&B\\\overline CP(t)&D}\bbm{x(t)\\u(t)},
\end{equation}
where again $\sbm{A_{-1}&B\\\overline C&D}$ is passive. The Lax-Phillips evolution family $\Tfrak_r$ of this system will be associated to the generator family $\Afrak_r(t):=\Afrak\Pfrak(t)+\Gfrak(t)$, see \eqref{eq:LPsgGen}, which equals
$$
\begin{aligned}
	&\Afrak_r(t)\bbm{y\\x_0\\u} =\bbm{y'\\\big(A_{-1}P(t)+G(t)\big)x_0+Bu(0)\\u'},  \\
	&\qquad\text{with}\quad \dom{\Afrak_r(t)}=\Pfrak(t)^{-1}\,\dom\Afrak,\quad t\in J.
\end{aligned}
$$
We have the following weaker analogue of Thm \ref{thm:Sigmal}:

\begin{theorem}\label{thm:Sigmar}
Let $\tau\in J$ be such that $\tau<\sup J$. There is a well-posed system $\Sigma_r=\sbm{\T_r&\Phi_r\\\Psi_r&\F_r}$ with time interval $J$ and Lax-Phillips evolution family $\Tfrak_r$, such that $\T_r$ has the properties asserted in Thm \ref{thm:generationR}, and moreover:
\begin{enumerate}
\item For $x_\tau\in X$ and $u\in L^2(J_\tau;X)$, the function
\begin{equation}\label{eq:GenSigmarState}
	x(t):=\T_r(t,\tau)x_\tau+\Phi_r(t,\tau)u,\quad t\in J_\tau,
\end{equation}
is in $H^1_{loc}(J_\tau;X_{-1})$ and it satisfies 
$$
	\dot x(t)=A_{-1,r}(t)\,x(t)+Bu(t)
$$
in $X_{-1}$ for almost all $t\in J_\tau$, where $A_{-1,r}$ is the main operator in \eqref{eq:ArSys}.

\item Every classical trajectory $(u,x,y)$ of \eqref{eq:ArSys} with time interval $J_\tau$ satisfies $P(\cdot)\,x(\cdot)\in C(J_\tau;Z)$, with $Z$ given in \eqref{eq:solspace}.

\item[3)] Every classical trajectory of \eqref{eq:ArSys} with time interval $J_\tau$ is uniquely determined by $x(\tau)$ and $u$, and it satisfies the power inequality \eqref{eq:SigmalPower} and the energy inequality \eqref{eq:SigmalEnergy}, both with equality if $\Sigma_i$ preserves energy.
\end{enumerate}

If $P(\cdot)z\in C^2(J;X)$ and $G(\cdot)z\in C^1(J;X)$ for all $z\in X$, then $\T_r$ has the properties asserted in Prop.\ \ref{prop:GenRsmooth}, and further:
\begin{enumerate}
\item[4)] The Lax-Phillips evolution family $\Tfrak_r$ of $\Sigma_r$ is generated by $\Afrak_r(t)$, $t\in J$, in the sense of Def.\ \ref{def:evolfam}.

\item[5)] For every $x_\tau$ and $u$ with $\big(P(\tau)\,x(\tau),u\big)\in V(\tau)$, there is a (unique) classical trajectory $(u,x,y)$ of \eqref{eq:ArSys} on $J_\tau$ with $x(\tau)=x_\tau$. The output satisfies $y\in H_{loc}^1(J_\tau;Y)$, and $(u,x,y)$ is also a trajectory of $\Sigma_r$. 
\end{enumerate}
\end{theorem}

\begin{proof}
The proof is similar to that of Thm \ref{thm:Sigmal}. First establish 3). Next, temporarily assume that $P(\cdot)z\in C^2(J;X)$ and $G(\cdot)z\in C^1(J;X)$ for all $z\in X$. Then use Prop.\ \ref{prop:GenRsmooth} instead of Thm \ref{thm:generationL} and $\big(P(\tau)\,x(\tau),u\big)\in V(\tau)$ instead of $\big(x(\tau),u\big)\in V(\tau)$, in step one of the proof of Thm \ref{thm:Sigmal}, to get items 4) and 5), apart from the claim that $(u,x,y)$ is also a trajectory of $\Sigma_r$. Item 2) is proved like the corresponding statement in Thm \ref{thm:Sigmal}.

The classical trajectory $(u,x,y)$ of \eqref{eq:ArSys} that was constructed in the previous paragraph is also a trajectory of $\Sigma_r$. Namely, by \eqref{eq:LPlInterp} for $\Tfrak_r$ and \eqref{eq:LPevol}, we get that
$$
\left\{\begin{aligned}
	y(t) &= 	\delta_t\Psi_r(t,\tau)x_\tau+\delta_t\F_r(t,\tau)\widetilde u , \\
	x(t) &= 	\T_r(t,\tau)x_\tau+\Phi_r(t,\tau)\widetilde u, \qquad t\in J_\tau;
\end{aligned}\right.
$$
then \eqref{eq:WPcausal}, \eqref{eq:WPlin} and the choice of $\widetilde u$ give
$$
\left\{\begin{aligned}
	x(t) &= 	\T_r(t,\tau)x_\tau+\Phi_r(t,\tau) u, \\
	\proj_{[\tau,t]}y &= \Psi_r(t,\tau)x_\tau+\F_r(t,\tau) u , \qquad t\in J_\tau.
\end{aligned}\right.
$$

Even without the additional smoothness assumptions, by Thm \ref{thm:Sigmal}, $\widetilde\Afrak_n(t):=\Pfrak(t)\Afrak+\Pfrak(t)\Gfrak_n(t)\Pfrak(t)^{-1}$, with domain $\dom\Afrak$, generates the Lax-Phillips evolution family $\widetilde\Tfrak_n$ of a well-posed system. Letting $n\to\infty$, we get from Thm \ref{thm:generationR} that there exists a unique evolution family $\Tfrak_r$ which satisfies
\begin{equation}\label{eq:TrLP}
\begin{aligned}
	&\Tfrak_r(t,\tau)w_0=
	\Pfrak(t)^{-1}\widetilde \Tfrak_l(t,\tau)\Pfrak(\tau)w_0 \\
	&\qquad+\int_\tau^t 
		\Pfrak(t)^{-1}\widetilde \Tfrak_l(t,s)\dot\Pfrak(s)
		\Tfrak_r(s,\tau) w_0\ud s
\end{aligned}
\end{equation}
for all $w_0\in \Hscr$, where $\widetilde\Tfrak_l$ is the strong limit of $\widetilde\Tfrak_n$. Since $\widetilde\Tfrak_n$ is associated to a well-posed system, so are $\widetilde\Tfrak_l$ and $\Tfrak_r$; see the proof of Thm \ref{thm:Sigmal}, step 3, and the proof of \cite[Prop.\ 4.3.a)]{SchWe10}.

Thm \ref{thm:generationR} and \cite[(4.17)]{SchWe10} give, for $u\in L^2(J;U)$,
$$
\begin{aligned}
&\frac{\partial}{\partial t}\Phi_r(t,\tau)u = 
	P(t)^{-1}\frac{\partial}{\partial t}\widetilde\Phi_l(t,\tau)u \\
	&\quad +\int_\tau^t A_{-1,r}(t)\T_r(t,s)P(s)^{-1}\dot P(s)P(s)^{-1} \widetilde\Phi_l(s,\tau)u\ud s,
\end{aligned}
$$
and using Thm \ref{thm:Sigmal} with \cite[(4.17)]{SchWe10}, we get that this equals $A_{-1,r}(t)\Phi_r(t,\tau)u+Bu(t)$ which is in $C(J_\tau;X_{-1})\subset L^2_{loc}(J_\tau;X_{-1})$ as a function of $t$. By Thm \ref{thm:generationR}, $\T_r(\cdot,\tau)x_\tau\in C^1(J_\tau;X_{-1})\subset H^1_{loc}(J_\tau;X_{-1})$ and $\frac{\partial}{\partial t}\T_r(t,\tau)x_\tau=A_{-1,r}(t)\T_r(t,\tau)x_\tau$. Hence, 1) holds.

The proof is complete once we have established that $\T_r$ equals the evolution family in Thm \ref{thm:generationR}; then $\T_r$ has the properties asserted in Thm \ref{thm:generationR}, and in the smooth case even those in Prop.\ \ref{prop:GenRsmooth}. For an arbitrary $x_\tau\in X$, we apply $\bbm{0&I&0}$ to \eqref{eq:TrLP} with $w_0:=\sbm{0\\I\\0}x_\tau$, from the left,
$$
\begin{aligned}
	&\T_r(t,\tau)x_\tau=
	P(t)^{-1}\widetilde \T_l(t,\tau)P(\tau)x_\tau \\
	&\qquad+\int_\tau^t 
		P(t)^{-1}\widetilde \T_l(t,s)\dot P(s)
		\T_r(s,\tau) x_\tau\ud s.
\end{aligned}	
$$
By \eqref{eq:PTtlP}, $\T_r$ satisfies \eqref{eq:TrInt}, and by uniqueness, $\T_r$ is the same evolution family as in Thm \ref{thm:generationR}.
\end{proof}

There are also representation formulas for multiplicative perturbation from the right, in case of smoother $P(\cdot)$ and $G(\cdot)$.

\begin{proposition}\label{prop:repR}
Assume that $P(\cdot)z\in C^2(J;X)$ and that $G(\cdot)z,\,G(\cdot)^*z\in C^1(J;X)$ for all $z\in X$. Let $\tau\in J$ with $\tau<\sup J$ and $t\in J_\tau$.

The operators $\T_r(t,\tau)$ have unique extensions to $\T_{-1,r}(t,\tau)\in\Lscr(X_{-1,r}^\tau;X_{-1,r}^t)$ which are locally uniformly bounded. The function $\tau\mapsto\T_r(t,\tau)x_0$, $\tau\in J$ with $\tau\leq t$, is continuously differentiable in $X_{-1}$, and
\begin{equation}\label{eq:dTrTau}
	\frac{\partial}{\partial\tau}\T_r(t,\tau)x_0=-\T_{-1,r}(t,\tau)A_{-1,r}(\tau)x_0.
\end{equation}

For all $x_\tau\in P(\tau)^{-1}\,\dom A$, 
$$
	\Psi_r(t,\tau)x_\tau=s\mapsto \overline CP(s)\,\T_r(s,\tau)x_\tau,\quad \tau\leq s\leq t,
$$
and for all $(0,u)\in V(\tau)$,
$$
\begin{aligned}
	\Phi_r(t,\tau)u &= \int_\tau^t \T_{-1,r}(t,s)Bu(s)\ud s \\
	\big(\F_r(t,\tau)u\big)(s) &= 
	\overline CP(s)\!\int_\tau^s \T_{-1,r}(s,\sigma)Bu(\sigma)\ud \sigma
		+Du(s),
\end{aligned}
$$ 
$\tau\leq s\leq t$, where the integrals are computed in $X_{-1}$. The representation formula for $\Phi_r$ in fact holds for all $u\in L_{loc}^2(J;U)$.
\end{proposition}

\begin{proof}
The claims on the extensions of $\T_r(t,\tau)$ follow as in the proof of Prop.\ \ref{prop:repL}, because the backward version of Thm \ref{thm:generationL} gives that $A_{r}^\dagger(t)^*=P(t)A^*+G(t)^*$ with $\dom{A_{r}^\dagger(t)}=\dom{A^*}$ generates the backward evolution family $\T_r^*$ with time interval $J$. Then also, for all $x_0\in X$ and $z\in\dom{A^*}$,
$$
\begin{aligned}
	\frac{\partial}{\partial \tau} \Ipdp{\T_r(t,\tau)x_0}{z}_{X_{-1},X_{1}^{d}} = 
	 \Ipdp{x_0}{\frac{\partial}{\partial \tau}\T_r(t,\tau)^*z}_{X_{-1},X_{1}^{d}}&=\\
	\Ipdp{-\T_{-1,r}(t,\tau)\big(A_{-1}P(\tau)+G(\tau)\big)x_0}{z}_{X_{-1},X_{1}^{d}} &,
\end{aligned}
$$
and we next prove that the weak derivative is strong. We have
$$
\begin{aligned}
	&\|\T_{-1,r}(t,\tau)\big(A_{-1}P(\tau)+G(\tau)\big)x_0 \\
	&\quad -\T_{-1,r}(t,\sigma)\big(A_{-1}P(\sigma)+G(\sigma)\big)x_0 \|_{-1} \leq \\
	&\|\T_{-1,r}(t,\tau)A_{-1}P(\tau)x_0-\T_{-1,r}(t,\sigma)A_{-1}P(\sigma)x_0 \|_{-1} \\
	&\quad+\|\T_{-1,r}(t,\tau)G(\tau)x_0-\T_{-1,r}(t,\sigma)G(\sigma)x_0 \|_{-1},
\end{aligned}
$$
where the first term tends to zero as $\sigma\to\tau$ by the proof of \cite[Prop.\ 2.7.c)]{SchWe10}, and the second term is at most
$$
\begin{aligned}
	&\|\big(\T_{r}(t,\tau)-\T_{r}(t,\sigma)\big)G(\tau)x_0 \|_{-1}\\
	&\quad+\|\T_{-1,r}(t,\sigma)\|\cdot\|G(\tau)x_0-G(\sigma)x_0 \|_{-1}\to0
\end{aligned}
$$
by the strong continuity of $\T_r$ and the local uniform boundedness of $\T_{-1,r}$. This proves \eqref{eq:dTrTau} and then the representation formulas can be proved the same way as the corresponding formulas in Prop.\ \ref{prop:repL}.
\end{proof}

\section{A time-varying wave equation}\label{sec:Wave}

This section is concerned with the time-varying wave equation, on a bounded $n$-dimensional Lipschitz domain, whose boundary $\partial\Omega$ has been split into a reflecting part $\Gamma_0$ and a part $\Gamma_1$ used for control and observation. We assume that $\Gamma_0$ and $\Gamma_1$ are relatively open with boundaries of measure zero within $\partial\Omega$. We do not make the restrictive assumption that $\overline\Gamma_0\cap\overline\Gamma_1=\emptyset$ or $\Gamma_0\neq0$ (since we do not need the Poincar\'e inequality). In this section, all vector spaces are real.

Written in so-called ``scattering form'', the boundary-controlled wave equation is (omitting the spatial variable):
\begin{equation}\label{eq:physPDEscatt}
  \left\{
    \begin{aligned}
        \rho(t)\,\ddot z (t) &= 
	  \Div T(t)\,\Grad z(t)-Q(t)\dot z(t)
	  \quad\text{on}~\Omega,\\
    \sqrt 2b\, u(t) &= \nu\cdot T(t)\,\Grad z(t) +b^2\,\dot z(t)
    	\qquad\quad\text{on}~\Gamma_1,\\
    \sqrt 2b\, y(t) &=  \nu\cdot T(t)\,\Grad z(t)-b^2\,\dot z(t)
	\qquad\quad\text{on}~\Gamma_1, \\
    0 &= \dot z(t)\qquad\qquad\qquad\qquad\text{on}~\Gamma_0,~ t\geq\tau,\\
    z(\tau) &= z_0,\quad \dot z(\tau) = z_1\qquad\qquad\qquad~ \text{on}~\Omega;
    \end{aligned}\right.
\end{equation}
here $z(t)=z(t,\xi)$ is the deflection at the point $\xi\in\overline\Omega$ at time $t$. The function $\rho(t,\cdot)\in L^\infty(\Omega)$ is the distributed mass density at time $t\in J$, which satisfies $\rho(t,\cdot)\geq\delta I$ for some $\delta>0$, $T(t,\cdot)\in L^\infty(\Omega)^{n\times n}$ is Young's modulus, with $T(t,\xi)=T(t,\xi)^*\geq \delta I$, and the operator $Q(t)\geq0$ describes viscous damping inside the domain $\Omega$. The scattering parameter $b\in L^\infty(\Gamma_1)$ is time-\emph{independent} with positive values a.e.\ and $\nu\in L^\infty(\partial\Omega;\R^n)$ is the outwards unit normal of $\partial\Omega$.

The time-varying mass density and internal viscous damper $Q$ bring our example beyond the theory of \cite{SchWe10}. Due to the time-varying $\rho$ and $T$, the example is not covered by \cite{ChenWeiss15} either.

In order to formulate the PDE \eqref{eq:physPDEscatt} in operator theory language, so that we can prove its well-posedness, we need to recall the setting of \cite{KuZwWave}. We equip $H^1(\Omega)$ with the graph norm of the gradient, and $H^{\mathrm{div}}(\Omega)$ is the space of all elements of $L^2(\Omega)^n$, whose (distribution) divergence lies in $L^2(\Omega)$, also equipped with the graph norm (see \cite{DaLiBook3}), so that these two spaces are Hilbert. For the definition of the fractional-order Hilbert space $H^{1/2}(\partial\Omega)$ on the boundary of $\Omega$, which is continuously embedded in $L^2(\partial\Omega)$, see \cite[\S 13.5]{TuWeBook}. Then
$$
	\Wscr := \set {h\in H^{1/2}(\partial\Omega)\bigmid h|_{\Gamma_0}=0}
$$ 
is Hilbert with the inherited norm, as the orthogonal projection onto $L^2(\Gamma_0)$ in $L^2(\partial\Omega)$ is bounded, $\Wscr$ is continuously and densely embedded in $L^2(\Gamma_1)$, \cite[Thm 13.6.10, (13.5.3)]{TuWeBook}.  We may then define $\Wscr'$ as the dual of $\Wscr$ with pivot space $L^2(\Gamma_1)$. 

The Dirichlet trace $\gamma_0$ maps $H^1(\Omega)$ continuously onto $H^{1/2}(\partial\Omega)$, and therefore the subspace
$$
	H^1_{\Gamma_0}(\Omega) := \set{g\in H^1(\Omega)\bigmid g\big|_{\Gamma_0}=0}
$$
is Hilbert with the norm inherited from $H^1(\Omega)$. Furthermore, $\gamma_0$ maps $H^1_{\Gamma_0}(\Omega)$ continuously onto $\Wscr$. In \cite[App.\ 1]{KuZwWave}, it was also shown that he restricted normal trace operator 
$$
	u\mapsto (\nu \cdot u)\big|_{\Gamma_1}:C^\infty(\overline\Omega)^n\to L^2(\Gamma_1)
$$ 
has a continuous extension $\gamma_\perp $ that maps $H^{\mathrm{div}}(\Omega)$ \emph{onto} $\Wscr'$. 

In order to write the wave equation as a ``physically motivated'' scattering passive system in the sense of \cite{StWe12b}, we denote 
$$
\begin{aligned}
	H&:=L^2(\Omega)^n,\quad  E:=L^2(\Omega),\quad 
		E_0:=H^1_{\Gamma_0}(\Omega), \\
	U&:=Y:=L^2(\Gamma_1), \qquad \text{and} \\
	L &:= -\Grad \big|_{E_0}\in\Lscr(E_0;H), \\
	K &:= \sqrt2b\,\gamma_0\big|_{E_0}\in\Lscr(E_0; U),\\
	G(t) &:= \bbm{0&0\\0&\dot\rho(t)/\rho(t)-Q(t)/\rho(t)}\in\Lscr(H\times E),\\
	P(t)&:=\bbm{T(t)&0\\0&1/\rho(t)}\in\Lscr(H\times E),
\end{aligned}
$$
where $G(t)$ and $P(t)$ are pointwise multiplication by the given functions of $\xi$, so that, e.g., 
\begin{equation}\label{eq:multipexplain}
	(P(t)x)(\xi)=P(t,\xi)\,x(\xi),\qquad\xi\in\Omega,
\end{equation}
for all $t\in J$ and $x\in H\times E=L^2(\Omega)^{n+1}$. Here $H$, $E$, and $U$ are identified with their duals. The space $E_0$ is densely and continuously contained in $E$. We take $E_0'$ to be the dual of $E_0$ with pivot space $E$ and denote the bounded dual of $L$ by $L'\in\Lscr\big(H;E_0')$. We can of course also consider $L$ as a densely defined unbounded operator on $E$ and in this case we would denote its unbounded adjoint on $E$ by $L^*$.

With the notation introduced above, we can interpret the PDE \eqref{eq:physPDEscatt} as a time-varying \emph{boundary control system} \cite{MaSt06}
\begin{equation}\label{eq:WaveBCS}
	\left\{\begin{aligned} \dot x(t)&=\mathfrak L(t)\,x(t), \\
		u(t)&=\mathfrak G(t)\,x(t),\\
		y(t)&=\mathfrak K(t)\,x(t),
\end{aligned}\right.
\quad t\geq\tau,~ x(\tau)=
	\bbm{\Grad z_0 \\ \rho(\tau)\,z_1},
\end{equation}
by introducing the state $x(t):=\sbm{\Grad z(t) \\ \rho(t)\,\dot z(t) }$ and the operators
$$
\begin{aligned}
	\mathfrak L(t)&:=\bbm{0&\Grad \\\Div&0}P(t)+G(t), \\
	\mathfrak G(t)&:=\frac{1}{\sqrt2 b}
		\bbm{\gamma_\perp&b^2\,\gamma_0}P(t), \\
	\mathfrak K(t)&:=\frac{1}{\sqrt2 b}
		\bbm{\gamma_\perp&-b^2\,\gamma_0}P(t),
\end{aligned}
$$
with the common domain 
\begin{equation}\label{eq:WaveSolSp}
Z(t):=P(t)^{-1}\bbm{H^{\rm div}(\Omega)\\H^1_{\Gamma_0}(\Omega)}.
\end{equation}
By \cite[\S4]{KuZwWave}, the boundary control system \eqref{eq:WaveBCS} is a (scattering) conservative boundary control system in case $P(t)=I$ and $G(t)=0$ for $t\geq\tau$. In this case, we omit $(t)$ from the notation as the system becomes time invariant, and then \cite[Thm 2.3]{MaSt06} gives that the mapping $\sbm{\mathfrak L\\\mathfrak K}\sbm{I\\\mathfrak G}^{-1}$ from $\sbm{x(t)\\u(t)}$ to $\sbm{\dot x(t)\\y(t)}$ in \eqref{eq:WaveBCS}, defined on $\sbm{I\\\mathfrak G}Z$, determines a conservative system,
\begin{equation}\label{eq:WaveIsoMS}
	\bbm{\dot x(t)\\y(t)} =\bbm{\mathfrak L\\\mathfrak K}
		\bbm{I\\\mathfrak G}^{-1}\bbm{x(t)\\u(t)},
	\quad t\geq\tau,~ x(\tau)=x_\tau,
\end{equation}
with the solution space $Z$ in \eqref{eq:WaveSolSp}, with $P(t)=I$.

We next recall the following integration by parts formula from \cite[(A5)]{KuZwWave}: for all $f\in H^{\mathrm{div}}(\Omega)$ and $g\in H^1_{\Gamma_0}(\Omega)$,
\begin{equation}\label{eq:intparts}
	\Ipdp{\Div f}{g}_{L^2(\Omega)}+\Ipdp{f}{\Grad g}_{L^2(\Omega)^n}
	=\Ipdp{\gamma_\perp f}{\gamma_0 g}_{\Wscr',\Wscr}.	
\end{equation}
With this formula, we can obtain the action of the bounded dual $L'$ on the dense subspace $H^{\mathrm{div}}(\Omega)$ of its domain $L^2(\Omega)^n$:

\begin{corollary}\label{cor:Ldual}
With $K_0:=\gamma_0\big|_{E_0}\in\Lscr(E_0;U)$, we have
\begin{equation}\label{eq:Ldualaction}
	L'f=\Div f-K_0'\gamma_\perp f,\qquad f\in H^{\mathrm{div}}(\Omega).
\end{equation}
Moreover, $K_0'$ is injective.
\end{corollary}

\begin{proof}
Using \eqref{eq:intparts}, for $w\in E_0=H^1_{\Gamma_0}(\Omega)$ and $f\in H^{\mathrm{div}}(\Omega)$:
$$
\begin{aligned}
	\Ipdp{w}{L'f+K_0'\,\gamma_\perp f}_{E_0,E_0'} &= \\
	\Ipdp{-\Grad w}{f}_{L^(\Omega)^n}+\Ipdp{\gamma_0\,w}{\gamma_\perp f}_{\Wscr',\Wscr} &= \\
	\Ipdp{w}{\Div f}_{L^2(\Omega)}= \Ipdp{w}{\Div f}_{E_0,E_0'},
\end{aligned}
$$
which implies \eqref{eq:Ldualaction}. Finally, if $K_0'g=0$ then
$$
	0=\Ipdp{h}{K_0'g}_{E_0,E_0'}=\Ipdp{\gamma_0 h}{g}_U
$$
for all $h\in E_0$, and consequently $g\in U\ominus \gamma_0E_0=\zero$.
\end{proof}

Cor.\ \ref{cor:Ldual} makes it easy to prove that \eqref{eq:WaveBCS} satisfies
\begin{equation}\label{eq:WaveWP}
\begin{aligned}
	&\bbm{\mathfrak L(t)\\ \mathfrak K(t)}
	\bbm{I\\\mathfrak G(t)}^{-1} = 
	\bbm{\,\overline\Ascr P(t)+G(t) & \Bscr\\\overline\Cscr P(t)&I}\bigg|_{\sbm{I\\\mathfrak G(t)}Z}, \\ 
	& \text{where} \qquad
	\overline\Ascr:=\bbm{0 & -L  \\ L' & -\frac12 K' K}, 
		\quad \Bscr:=\bbm{0\\ K'},\\
	& \overline\Cscr:=\bbm{0 & -K},\qquad
	\dom{\overline \Ascr}=
	\dom{\overline \Cscr}=\bbm{H\\E_0}.
\end{aligned}
\end{equation}
Then $\Ascr\subset \overline\Ascr\subset \Ascr_{-1}$, where (with $X:=H\times E$)
$$
	\Ascr:=\overline\Ascr\big|_{\dom{\Ascr}},\qquad
	\dom{\Ascr}:=\set{x\in X\mid \overline\Ascr x\in X},
$$
and $\Ascr_{-1}$ is the unique extension of $\Ascr$ to an operator in $\Lscr(X;X_{-1})$; see \S\ref{sec:prel} for details. 

The unbounded operator
$$
	\bbm{L\\K}:E\supset\dom{\sbm{L\\K}}=E_0\to\bbm{H\\U}
$$ 
is closed, since the Hilbert space $E_0$ is equipped with the graph norm of $-L$ and $K\in\Lscr(E_0;U)$. In case $T(t)=I$, $\rho=1$ and $Q(t)=Q\geq 0$ for all $t\geq0$, then 
\begin{equation}\label{eq:WaveSysFinal}
	\bbm{\dot x(t)\\y(t)}=
	\bbm{\,\overline\Ascr P(t)+G(t) & \Bscr\\\overline\Cscr P(t)&I}	\bbm{x(t)\\u(t)},\quad t\geq\tau,
\end{equation}
with the operators defined in \eqref{eq:WaveWP} and $G(t)=\sbm{0&0\\0&-Q}$, is indeed in the ``physically motivated'' class, with state space $X=H\times E$, by \cite[Thm 1.1]{StWe12b}. The connection between \eqref{eq:WaveSysFinal} and \eqref{eq:physPDEscatt} is made more precise in the next result, which states that the equations have the same classical solutions. 

\begin{theorem}\label{thm:WaveEq}
Assume that $T$ and $T^{-1}$ are such that the functions $T(\cdot)z=T(\cdot)^*z$ are in $C^1\big(J;L^2(\Omega)^n\big)$ for all $z\in L^2(\Omega)^n$, cf.\ \eqref{eq:multipexplain}. Also, assume that the multiplications by $\rho(\cdot)$ and $1/\rho(\cdot)$ are strongly in $C^1\big(J;L^2(\Omega)\big)$, and that $Q(\cdot)$ is strongly in $C\big(J;L^2(\Omega)\big)$. Then $\Sigma_w$ in \eqref{eq:WaveSysFinal} is a well-posed system with state space $X=L^2(\Omega)^n\times L^2(\Omega)$, input/output space $U=L^2(\Gamma_1)$, and properties 1)\,--\,3) in Theorem \ref{thm:Sigmar}. 

Now additionally assume that $T(\cdot)$ and $\rho(\cdot)$ are strongly in $C^2$ and that $Q(\cdot)$ is strongly in $C^1$. Then $\Sigma_w$ has properties 4)\,--\,5) in Thm \ref{thm:Sigmar}. Moreover, for every $z_0,z_1\in H^1_{\Gamma_0}(\Omega)$, $\tau\in J$ with $\tau<\sup J$, and $u\in H^1_{loc}(J_\tau;U)$, such that
$$
\begin{aligned}
	T(\tau)\Grad z_0 &\in H^{\rm div}(\Omega)
		\qquad\text{and}\\
	\gamma_\perp T(\tau)\Grad z_0 &+ b^2\gamma_0z_1
		=\sqrt 2b\,u(\tau),
\end{aligned}
$$
there is a unique solution of \eqref{eq:physPDEscatt} on $J_\tau$ that satisfies 
\begin{equation}\label{eq:WaveSmooth}
	z\in C^2\big(J_\tau; L^2(\Omega)\big)\cap 
		C^1\big(J_\tau;H^1_{\Gamma_0}(\Omega)\big).
\end{equation}
For this solution, $\left(u,x,y\right)$ is a classical trajectory of \eqref{eq:WaveSysFinal} on $J_\tau$, where $x(t):=\sbm{\Grad z(t) \\ \rho(t)\,\dot z(t) }$, $t\in J_\tau$, the output signal satisfies $y\in H^1_{loc}(J_\tau;Y)$, and the power balance
\begin{equation}\label{eq:WavePow}
\begin{aligned}
	&\ddt  \Ipdp{T(t)\Grad z(t)}{\Grad z(t)}_H
	+ \ddt \Ipdp{\rho(t)\dot z(t)}{\dot z(t)}_E \\
	&\quad + \|y(t)\|^2_U = \|u(t)\|^2_U
	+ \Ipdp{\dot T(t)\Grad z(t)}{\Grad z(t)}_H \\
	&\qquad\qquad +
	\Ipdp{\dot\rho(t)\dot z(t)}{\dot z(t)}_E
	-2\Ipdp{Q(t)\dot z(t)}{\dot z(t)}_E
\end{aligned}
\end{equation}
holds for $t\in J_\tau$, together with the corresponding integrated energy balance. In this case, if $Q(\cdot)$ is a multiplication operator like \eqref{eq:multipexplain}, then the assertions in Prop.\ \ref{prop:repR} hold.
\end{theorem}

A multiplication operator $M(\cdot)$ of the type \eqref{eq:multipexplain} satisfies $M(\cdot) f\in C^\ell \big(J;L^2(\Omega)^k\big)$ for all $f\in L^2(\Omega)^k$ if all its component functions $M_{i,j}(\cdot)\in C^\ell\big(J;C(\overline\Omega)\big)$. If additionally $M(t,\xi)\geq \delta I$ for some $\delta>0$ independent of $t\in J$ and $\xi\in \overline\Omega$, then also $M(\cdot)^{-1}f\in C^\ell\big(J;L^2(\Omega)^k\big)$ for all $f\in L^2(\Omega)^k$; see \eqref{eq:PinvDiff}. Thus, even the smoother assumptions on $T(\cdot)$, $\rho(\cdot)$ and $Q(\cdot)$ in Thm \ref{thm:WaveEq} are satisfied in the following interesting particular case, cf.\ \cite[pp.\ 177--181]{ChenWeiss15}: 

A rigid object moves inside the domain $\Omega$, with center point $\eta(t)$ moving according to $\ddot\eta(t)=a\big(t,\eta(t)\big)$, where the acceleration field $a\in C(J\times\overline\Omega;\R^n)$; then $\eta\in C^2(J;\Omega)$. In the moving object, the physical parameters $\rho(t,\xi)$, $T(t,\xi)$ and $Q(t,\xi)$ depend only on the distance to the center point, in a twice continuously differentiable manner; they are of the form
$$
	(M(t)f)(\xi)=m(\|\eta(t)-\xi\|^2)\,f(\xi),\quad t\in J,~\xi\in\Omega,
$$ 
with $m\in C^2(\R)^{k\times k}$ such that $m(\cdot)\geq \delta I$ for some $\delta>0$.

\begin{proof}
For the first assertion, use the first, less smooth part of Thm \ref{thm:Sigmar} on \eqref{eq:WaveSysFinal}, observing that $P(\cdot)$ and $G(\cdot)$ satisfy the standing assumptions \eqref{eq:standing}. For the last statement, note that if $Q(\cdot)$ is a multiplication operator then it is self-adjoint, and hence $G(\cdot)^*$ inherits strong continuity from $G(\cdot)$. It suffices to prove the rest of the statements for compact intervals $[a,b]\subset J$; then, in particular, $H^1([a,b];U)=H^1_{loc}([a,b];U)$. Assume that $T(\cdot),\,\rho(\cdot)\in C^2$ and $Q(\cdot)\in C^1$, strongly.

Let $z_0$, $z_1$ and $u$ be as in the second assertion. Then $x_\tau:=\sbm{\Grad z_0\\\rho(\tau)z_1}$ satisfies $(P(\tau) \, x_\tau,u)\in V(\tau)$, since $\overline\Ascr\subset \Ascr_{-1}$, and
$$
	\Ascr_{-1} P(\tau)x_\tau+\Bscr u(\tau) =
	\bbm{\Grad z_1 \\\Div \big(T(\tau)\Grad z_0\big)}
	\in \bbm{L^2(\Omega)^n\\L^2(\Omega)}=X.
$$
Since multiplication by $1/\rho(\cdot)$, which is the inverse operator of multiplication by $\rho(\cdot)$, is strongly in $C^1\big([a,b];L^2(\Omega)\big)$, together with $Q(\cdot)$ and multiplication by $\dot\rho(\cdot)$, the operator function $G(\cdot)$ is strongly in $C^1([a,b];X)$. Using \eqref{eq:PinvDiff}, one obtains that $1/\rho(\cdot)$ is strongly in $C^2$, so that $P(\cdot)$ is strongly in $C^2([a,b];X)$. By the smoother part of Thm \ref{thm:Sigmar}, there exist unique $x$ and $y$, such that $(u,x,y)$ is a classical trajectory of \eqref{eq:WaveSysFinal} with $x(\tau)=x_\tau$, and this trajectory satisfies \eqref{eq:SigmalPower} with equality, $y\in H^1(J_\tau;Y)$ and $P(\cdot)\,x(\cdot)\in C(J_\tau;Z)$. Writing $x(t)=:\sbm{x_1(t)\\x_2(t)}$ and defining 
\begin{equation}\label{eq:WavezDef}
	z(t):=z_0+\int_\tau^t \frac{x_2(s)}{\rho(s)}\ud s,
\end{equation}
we get $x(t)=\sbm{\Grad z(t)\\\rho(t)\dot z(t)}$, and then \eqref{eq:SigmalPower} specializes to \eqref{eq:WavePow}. The definition of classical trajectory now gives that
$$
	x(t),\,\dot x(t)=\bbm{\Grad \dot z(t)\\\dot \rho(t)\dot z(t)
	+\rho(t)\ddot z(t)}
	\in C\left(J_\tau; \bbm{L^2(\Omega)^n\\ L^2(\Omega)}\right).
$$
Thus \eqref{eq:WaveSmooth} holds with $H^1(\Omega)$ instead of $H^1_{\Gamma_0}(\Omega)$. 

We next prove that $(u,z,y)$ is a solution of \eqref{eq:physPDEscatt} in the $L^2$ sense. Using the formula for $\dot x(t)$, \eqref{eq:WaveWP} and Cor.\ \ref{cor:Ldual}, we get
$$
\begin{aligned}
	\rho(t)\ddot z(t)
	&= L'T(t)\Grad z(t)-\frac12 K'K\dot z(t)+K'u(t) -Q(t)\dot z(t) \\
	&= \Div T(t)\Grad z(t) -Q(t)\dot z(t)
		- K_0'\gamma_\perp T(t)\Grad z(t) \\
	&\qquad -K_0'\,b^2 \gamma_0 \dot z(t)+K_0'\sqrt2b\, u(t) ,
\end{aligned}
$$
as an equality in $E_0'$. Applying this functional to an arbitrary test function, $\varphi\in C^\infty(\Omega)$ with compact support in $\Omega$, we get that $\rho(t)\ddot z(t)+Q(t)\dot z(t)=\Div T(t)\Grad z(t) $ in the sense of distributions. 
From $P(\cdot)\,x(\cdot)\in C(J_\tau;Z)$ and \eqref{eq:WaveSolSp}, we moreover get that $\dot z(t)\in H^1_{\Gamma_0}(\Omega)$ and $T(t)\Grad z(t)\in H^{\rm div}(\Omega)$, so that in fact $\rho(t)\ddot z(t)=\Div T(t)\Grad z(t)-Q(t)\dot z(t)$ in $L^2(\Omega)$, i.e., the first and fourth lines of \eqref{eq:physPDEscatt} are satisfied in $L^2$. Then the injectivity of $K_0'$ and the preceding display gives line 2 of \eqref{eq:physPDEscatt}, as an equality in $L^2(\Gamma_1)$, and line 3 follows from \eqref{eq:WaveWP}:
$$
	\sqrt2b\,y(t)=\gamma_\perp\big(T(t)\,\Grad z(t)\big) 
		+b^2\,\gamma_0\dot z(t)-2b^2\,\gamma_0\dot z(t),
$$
in $L^2(\Gamma_1)$. By \eqref{eq:WavezDef}, we have $z(\tau)=z_0$ and $\dot z(\tau)=x_2(\tau)/\rho(\tau)=z_1$. 

It remains to prove uniqueness. Let $(u,z,y)$ be a classical solution of \eqref{eq:physPDEscatt} with \eqref{eq:WaveSmooth}. Defining $x(t):=\sbm{\Grad z(t)\\\rho(t)\dot z(t)}$, we get from the calculations above that $(u,x,y)$ is a classical trajectory of \eqref{eq:ArSys} with $x(\tau)=\sbm{\Grad z_0\\\rho(\tau)z_1}$ and the given input signal $u$. By Thm \ref{thm:Sigmar}, such a trajectory is unique.
\end{proof}

The detour via \eqref{eq:WaveBCS} was needed only in order to establish that the solution space of \eqref{eq:WaveSysFinal} with $P(t)=I$ and $G(t)=0$ is $Z$ in \eqref{eq:WaveSolSp}, which in turn was needed to prove that $\dot z(t)\big|_{\Gamma_0}=0$ for $t\geq\tau$. We can unfortunately not prove a complete analogue of \cite[Thm 6.3]{ChenWeiss15} for mild trajectories of \eqref{eq:WaveSysFinal}, as we were unable to establish  \eqref{eq:SigmalEnergy} for mild trajectories in Thm \ref{thm:Sigmar}.

\section{Acknowledgment}

The author thanks George Weiss for suggesting this very interesting research project. The research was partially funded by the foundation of Ruth and Nils-Erik Stenb\"ack.


\begin{thebibliography}{10}
\providecommand{\url}[1]{#1}
\csname url@samestyle\endcsname
\providecommand{\newblock}{\relax}
\providecommand{\bibinfo}[2]{#2}
\providecommand{\BIBentrySTDinterwordspacing}{\spaceskip=0pt\relax}
\providecommand{\BIBentryALTinterwordstretchfactor}{4}
\providecommand{\BIBentryALTinterwordspacing}{\spaceskip=\fontdimen2\font plus
\BIBentryALTinterwordstretchfactor\fontdimen3\font minus
  \fontdimen4\font\relax}
\providecommand{\BIBforeignlanguage}[2]{{%
\expandafter\ifx\csname l@#1\endcsname\relax
\typeout{** WARNING: IEEEtran.bst: No hyphenation pattern has been}%
\typeout{** loaded for the language `#1'. Using the pattern for}%
\typeout{** the default language instead.}%
\else
\language=\csname l@#1\endcsname
\fi
#2}}
\providecommand{\BIBdecl}{\relax}
\BIBdecl

\bibitem{SchWe10}
R.~Schnaubelt and G.~Weiss, ``Two classes of passive time-varying well-posed
  linear systems,'' \emph{Math. Control Signals Systems}, vol.~21, no.~4, pp.
  265--301, 2010.

\bibitem{ChenWeiss15}
\BIBentryALTinterwordspacing
J.-H. Chen and G.~Weiss, ``Time-varying additive perturbations of well-posed
  linear systems,'' \emph{Math. Control Signals Systems}, vol.~27, no.~2, pp.
  149--185, 2015. 
\BIBentrySTDinterwordspacing

\bibitem{KuZwWave}
\BIBentryALTinterwordspacing
M.~Kurula and H.~Zwart, ``Linear wave systems on {$n$}-{D} spatial domains,''
  \emph{Internat. J. Control}, vol.~88, no.~5, pp. 1063--1077, 2015. 
\BIBentrySTDinterwordspacing

\bibitem{StWe12b}
O.~J. Staffans and G.~Weiss, ``A physically motivated class of scattering
  passive linear systems,'' \emph{SIAM Journal on Control and Optimization},
  vol. 50(5), pp. 3083--3112, 2012.

\bibitem{StWe12a}
\BIBentryALTinterwordspacing
G.~Weiss and O.~J. Staffans, ``Maxwell's equations as a scattering passive
  linear system,'' \emph{SIAM J. Control Optim.}, vol.~51, no.~5, pp.
  3722--3756, 2013. [Online]. Available:
  \url{http://dx.doi.org/10.1137/120869444}
\BIBentrySTDinterwordspacing

\bibitem{JaZwBook}
B.~Jacob and H.~Zwart, \emph{Linear port-Hamiltonian systems on
  infinite-dimensional spaces}, ser. Operator Theory: Advances and
  Applications.\hskip 1em plus 0.5em minus 0.4em\relax {Birkhäuser}-Verlag,
  2012, vol. 223.

\bibitem{GZM05}
Y.~L. Gorrec, H.~Zwart, and B.~Maschke, ``Dirac structures and boundary control
  systems associated with skew-symmetric differential operators,'' \emph{SIAM
  J. Control Optim.}, vol.~44, no.~5, pp. 1864--1892, 2005.

\bibitem{JaLa19}
B.~Jacob and H.~Laasri, ``Well-posedness of infinite-dimensional non-autonomous
  passive boundary control systems,'' 2019,
  \url{https://arxiv.org/abs/1901.11348}.

\bibitem{Pau17}
\BIBentryALTinterwordspacing
L.~Paunonen, ``Robust output regulation for continuous-time periodic systems,''
  \emph{IEEE Trans. Automat. Control}, vol.~62, no.~9, pp. 4363--4375, 2017.
  [Online]. Available: \url{https://doi.org/10.1109/TAC.2017.2654968}
\BIBentrySTDinterwordspacing

\bibitem{PauPoh12}
\BIBentryALTinterwordspacing
L.~Paunonen and S.~Pohjolainen, ``Periodic output regulation for distributed
  parameter systems,'' \emph{Math. Control Signals Systems}, vol.~24, no.~4,
  pp. 403--441, 2012. [Online]. Available:
  \url{https://doi.org/10.1007/s00498-012-0087-x}
\BIBentrySTDinterwordspacing

\bibitem{StafBook}
O.~J. Staffans, \emph{Well-Posed Linear Systems}.\hskip 1em plus 0.5em minus
  0.4em\relax Cambridge and New York: Cambridge University Press, 2005.

\bibitem{EnNa00}
K.-J. Engel and R.~Nagel, \emph{One-parameter semigroups for linear evolution
  equations}, ser. Graduate Texts in Mathematics.\hskip 1em plus 0.5em minus
  0.4em\relax New York: Springer-Verlag, 2000, vol. 194.

\bibitem{MaStWe06}
J.~Malinen, O.~J. Staffans, and G.~Weiss, ``When is a linear system
  conservative?'' \emph{Quarterly Appl.\ Math.}, vol.~64, pp. 61--91,
  2006.

\bibitem{Rudin73}
W.~Rudin, \emph{Functional analysis}.\hskip 1em plus 0.5em minus 0.4em\relax
  New York: McGraw-Hill Book Co., 1973, mcGraw-Hill Series in Higher
  Mathematics.

\bibitem{TanabeBook}
H.~Tanabe, \emph{Equations of evolution}, Pitman, 1979.

\bibitem{PazyBook}
A.~Pazy, \emph{Semi-Groups of Linear Operators and Applications to Partial
  Differential Equations}.\hskip 1em plus 0.5em minus 0.4em\relax Berlin:
  Springer-Verlag, 1983.

\bibitem{How74}
\BIBentryALTinterwordspacing
J.~S. Howland, ``Stationary scattering theory for time-dependent
  {H}amiltonians,'' \emph{Math. Ann.}, vol. 207, pp. 315--335, 1974. [Online].
  Available: \url{https://doi.org/10.1007/BF01351346}
\BIBentrySTDinterwordspacing

\bibitem{ChLa99}
\BIBentryALTinterwordspacing
C.~Chicone and Y.~Latushkin, \emph{Evolution semigroups in dynamical systems
  and differential equations}, ser. Mathematical Surveys and Monographs.\hskip
  1em plus 0.5em minus 0.4em\relax American Mathematical Society, Providence,
  RI, 1999, vol.~70. [Online]. Available:
  \url{https://doi.org/10.1090/surv/070}
\BIBentrySTDinterwordspacing

\bibitem{CuPrBook}
R.~F. Curtain and A.~J. Pritchard, \emph{Infinite Dimensional Linear Systems
  Theory}, ser. Lecture Notes in Control and Information Sciences.\hskip 1em
  plus 0.5em minus 0.4em\relax New York and Berlin: Springer-Verlag, 1978,
  vol.~8.

\bibitem{DaLiBook3}
R.~Dautray and J.-L. Lions, \emph{Mathematical analysis and numerical methods
  for science and technology. {V}ol. 3}.\hskip 1em plus 0.5em minus 0.4em\relax
  Springer, 1990, spectral theory and applications.

\bibitem{TuWeBook}
\BIBentryALTinterwordspacing
M.~Tucsnak and G.~Weiss, \emph{Observation and control for operator
  semigroups}, {Birkhäuser}, 2009.
\BIBentrySTDinterwordspacing

\bibitem{MaSt06}
J.~Malinen and O.~J. Staffans, ``Conservative boundary control systems,''
  \emph{Journal of Differential Equations}, vol. 231, pp. 290--312, 2006.

\end{thebibliography}

\def\cprime{$'$}

\begin{IEEEbiography}
[{\includegraphics[width=1in,height=1.25in,clip,keepaspectratio]{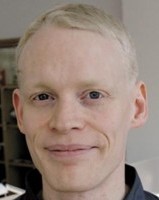}}]
{Mikael Kurula}
obtained his M.Sc. and Ph.D. degrees from \AA bo Akademi University, Finland, in 2004 and 2010, respectively. Since 2014 he is a senior lecturer at the Department of Mathematics and Statistics at the same university. His research focuses on systems and control theory for infinite-dimensional linear systems.
\end{IEEEbiography}

\end{document}